\documentclass[12pt]{amsart}
\usepackage[margin=1.42in]{geometry}

\usepackage{amsmath,amssymb,amsthm}
\usepackage{paralist}
\usepackage{graphics}
\usepackage{epsfig}
\usepackage{mathrsfs}
\usepackage{graphicx}
\usepackage[all]{xy}
\usepackage{subfigure}
\usepackage[usenames,dvipsnames,svgnames,table]{xcolor}
\usepackage{tikz}
\usetikzlibrary{matrix,arrows,backgrounds,calc,chains,automata,positioning,patterns}
\usetikzlibrary{shapes.geometric,calc}
\usepackage{enumerate}
\usepackage{bbold}
\usepackage{epstopdf}
\usepackage[colorlinks=true]{hyperref}
\hypersetup{urlcolor=blue, citecolor=red}
\usepackage{cite}

\newtheorem{theorem}{Theorem}[section]

\newtheorem{lemma}[theorem]{Lemma}
\theoremstyle{definition}

\newtheorem{remark}{Remark}

\newtheorem*{thm*}{Theorem}

\newtheorem{defn}{Definition}[section]

\theoremstyle{plain}

\newtheorem{cor}[defn]{Corollary}

\newtheorem{prop}[defn]{Proposition}

\newtheorem{thm}[defn]{Theorem}

\newtheorem{conjecture}{Conjecture}

\theoremstyle{definition}

\newcommand{\set}[1]{\left\{#1\right\}}

\newcommand{\paren}[1]{\left(#1\right)}

\newcommand{\R}{\mathbb{R}}
\newcommand{\eng}{\mathscr{E}}

\newcommand{\E}{\mathsf{E}}
\newcommand{\T}{\mathsf{T}}
\newcommand{\flux}{\operatorname{flux}}

\newcommand{\inviz}[1]{#1}

\def\s-s{self-similar}

\renewcommand{\div}{\operatorname{div}}
\numberwithin{equation}{section}
	\title[Barlow--Bass resistance estimates for hexacarpet]{Dual graphs and modified  Barlow--Bass resistance estimates for repeated barycentric subdivisions}
	
\author[D. J. Kelleher, H. Panzo, A. Brzoska and A. Teplyaev]{}

\email{dkellehe@mtholyoke.edu}
\email{hugo.panzo@uconn.edu}
\email{antoni.brzoska@uconn.edu}
\email{alexander.teplyaev@uconn.edu}


\urladdr{\url{https://djkelleher.wordpress.com/}}
\urladdr{\url{http://www.math.uconn.edu/~panzo/}}
\urladdr{\url{http://www.math.uconn.edu/~teplyaev/}}
\thanks{Research supported in part by NSF grants  DMS 1106982, 1262929, 1613025.}
\begin{document}
	\maketitle

\centerline{\scshape Daniel J.~Kelleher}

\medskip
{\footnotesize
	\centerline{Department of Mathematics and Statistics		
		}
	\centerline{Mount Holyoke College,
		South Hadley, MA 01075, USA}
}

\medskip

\centerline{\scshape Hugo Panzo, Antoni Brzoska and Alexander Teplyaev}

\medskip
{\footnotesize
	\centerline{Department of Mathematics}
	\centerline{University of Connecticut, Storrs, CT 06269, USA}
}

%

\begin{abstract}
We prove Barlow--Bass type resistance estimates for two random walks associated with repeated barycentric subdivisions of a triangle. If the random walk jumps between the centers of  triangles in the subdivision that have common sides, the resistance scales as a power of a constant $\rho$ which is theoretically estimated to be in the interval $5/4\leqslant\rho\leqslant3/2$, with a  numerical estimate $\rho\approx1.306$. This corresponds to the theoretical estimate of spectral dimension $d_S$ between 1.63 and 1.77, with a  numerical estimate $d_S\approx1.74$.  On the other hand, if the random walk jumps between the corners of  triangles in the subdivision, then the resistance  scales as a power of  a constant $\rho^T=1/\rho$, which is theoretically estimated to be in the interval $2/3\leqslant\rho^T\leqslant4/5$. This  corresponds to the spectral dimension between 2.28 and 2.38.
The difference between
$\rho$ and $\rho^T$ implies that the
the limiting behavior of random walks on the
repeated barycentric subdivisions is more delicate than   on  the
generalized Sierpinski Carpets, and suggests interesting possibilities for further research, including possible non-uniqueness of self-similar Dirichlet forms.
%
\end{abstract}

\section{Introduction}\label{sec-intro}

There has been an wide interest in studying analysis and random processes on various metric measure spaces that satisfy either the volume doubling property, or curvature bounds, or both.
One part  of this very large literature deals with spaces  where Lipchitz functions can be analyzed. Without even attempting to suggest a representative sample of relevant papers, we briefly mention such recent works as
\cite{HLTV16, Am, book, BK16}.
Another kind of more probabilistically inspired analysis   deals with
spaces that are more fractal in nature and have sub-Gaussian heat kernel estimates
(see \cite{Barlow,BB90,BBS90,BB99,BBKT,Bass} and references therein). Typically for such fractal examples, Lipschitz functions play little or no role, as intrinsically smooth functions are only H\"older continuous. In some sense all these results are related to the Nash-Moser theory of uniformly elliptic operators. However, there are natural spaces that have no volume doubling, no curvature bounds, and no heat kernel estimates. Analysis of such spaces is in its infancy, and considering even simplest examples is very challenging.
After laying some of the initial framework for this model, our aim is to connect to a series of other works, such as
\cite{SteinhurstBond,SteinhurstPOTA,HSTZjulia,KSWpcf,SteinhurstQM,KW1,KW2}.

\begin{figure}[t]
\inviz{\begin{tikzpicture}[scale=.5000]

\draw (90:4)--(210:4)--(330:4) -- cycle;

\end{tikzpicture}
\
\begin{tikzpicture}[scale=.5000]

\foreach \a in {0,1,2}{
\draw ($(90+120*\a:4)$)--($(-90+120*\a:2)$);
}

\draw (90:4)--(210:4)--(330:4) -- cycle;
\end{tikzpicture}
\
\begin{tikzpicture}[scale=.5000]

\foreach \a in {0,1,2}{
\draw ($(90+120*\a:4)$)--($(-90+120*\a:2)$);
\draw (0:0)--($.5*(90+120*\a:4)+.5*(30+120*\a:2)$);
\draw (0:0)--($.5*(-30+120*\a:4)+.5*(30+120*\a:2)$);
\draw ($(90+120*\a:4)$)--($(30+120*\a:1)$);
\draw ($(90+120*\a:4)$)--($(150+120*\a:1)$);
\draw ($(-30+120*\a:2)$)--($(-90+120*\a:2)$);
\draw ($(210+120*\a:2)$)--($(-90+120*\a:2)$);
}

\draw (90:4)--(210:4)--(330:4) -- cycle;
\end{tikzpicture}}
\caption{Barycentric subdivision of a 2-simplex, the graphs $G_0^T$, $G_1^T$ and $G_2^T$.}
\label{barycentricsubdivision1}
\end{figure}
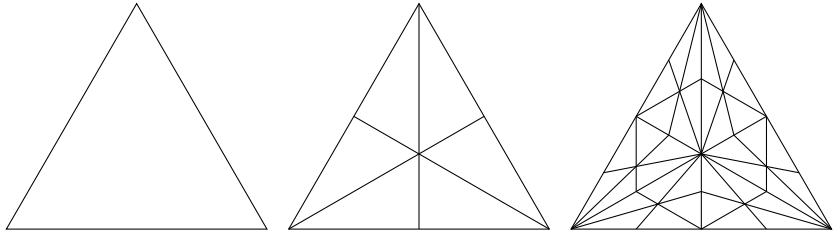
\begin{figure}
\begin{tikzpicture}[scale=.9000]
\draw 
(90:4cm)--(210:4cm)--(330:4cm) -- cycle;

\foreach \a in {0,1,2}{
\draw 
($(120*\a+90:4cm)$) -- ($(120*\a-90:2cm)$);
}

\draw 
(30:1cm)--(90:4cm)--(150:1cm)--(210:4cm)--(270:1cm)--(330:4cm) -- cycle;
\draw 
(-30:2cm)--(-90:2cm)--(-150:2cm)--(-210:2cm)--(-270:2cm)--(-330:2cm) -- cycle;

\foreach \a in{0,2,4}{
\draw (0,0)--($(60*\a-10:2.6cm)$);
}

\foreach \a in{1,3,5}{
\draw 
(0,0)--($(60*\a+10:2.6cm)$);
}

\foreach \a in{0,2,4}{
\draw[fill=black] ($.111*(150+60*\a:2cm)+.111*(90+60*\a:4cm) + .333*(90 + 60*\a:2cm) $) circle (3.3333pt);
\draw[fill=black] ($.111*(150+60*\a:2cm)+.111*(90+60*\a:4cm) + .333*(90 + 60*\a:2cm)+.333*(90 + 60*\a:4cm) $) circle (3.3333pt);
\draw[fill=black] ($.111*(150+60*\a:2cm)+.111*(90+60*\a:4cm) + .333*(90 + 60*\a:4cm) +.333*(60*\a+110:2.6cm)$) circle (3.3333pt);

\draw[fill=black] ($.111*(150+60*\a:2cm)+.111*(90+60*\a:4cm) + .333*(150 + 60*\a:1cm) $) circle (3.3333pt);
\draw[fill=black] ($.111*(150+60*\a:2cm)+.111*(90+60*\a:4cm) + .333*(150 + 60*\a:2cm) + .333*(150 + 60*\a:1cm) $) circle (3.3333pt);
\draw[fill=black] ($.111*(150+60*\a:2cm)+.111*(90+60*\a:4cm) + .333*(150 + 60*\a:2cm) +.333*(60*\a+110:2.6cm)$) circle (3.3333pt);
}
\foreach \a in{0,2,4}{
\draw [very thick]
($.111*(150+60*\a:2cm)+.111*(90+60*\a:4cm) + .333*(90 + 60*\a:2cm) $)--
($.111*(150+60*\a:2cm)+.111*(90+60*\a:4cm) + .333*(90 + 60*\a:2cm) + .2*(0 + 60*\a:2cm) $)-- %
($.111*(150+60*\a:2cm)+.111*(90+60*\a:4cm) + .333*(90 + 60*\a:2cm)  $)-- %
 ($.111*(150+60*\a:2cm)+.111*(90+60*\a:4cm) + .333*(90 + 60*\a:2cm)+.333*(90 + 60*\a:4cm) $) --
 ($.111*(150+60*\a:2cm)+.111*(90+60*\a:4cm) + .333*(90 + 60*\a:2cm)+.333*(90 + 60*\a:4cm) + .2*(0 + 60*\a:2cm) $) -- %
 ($.111*(150+60*\a:2cm)+.111*(90+60*\a:4cm) + .333*(90 + 60*\a:2cm)+.333*(90 + 60*\a:4cm) $) -- %
 ($.111*(150+60*\a:2cm)+.111*(90+60*\a:4cm) + .333*(90 + 60*\a:4cm) +.333*(60*\a+110:2.6cm)$)  --
($.111*(150+60*\a:2cm)+.111*(90+60*\a:4cm) + .333*(150 + 60*\a:2cm) +.333*(60*\a+110:2.6cm)$) --
 ($.111*(150+60*\a:2cm)+.111*(90+60*\a:4cm) + .333*(150 + 60*\a:2cm) + .333*(150 + 60*\a:1cm) $) --
 ($.111*(150+60*\a:2cm)+.111*(90+60*\a:4cm) + .333*(150 + 60*\a:2cm) + .333*(150 + 60*\a:1cm) + .4*(240 + 60*\a:2cm) $) -- %
 ($.111*(150+60*\a:2cm)+.111*(90+60*\a:4cm) + .333*(150 + 60*\a:2cm) + .333*(150 + 60*\a:1cm) $) -- %
 ($.111*(150+60*\a:2cm)+.111*(90+60*\a:4cm) + .333*(150 + 60*\a:1cm) $) --
 ($.111*(150+60*\a:2cm)+.111*(90+60*\a:4cm) + .333*(150 + 60*\a:1cm) + .4*(240 + 60*\a:2cm) $) -- %
 ($.111*(150+60*\a:2cm)+.111*(90+60*\a:4cm) + .333*(150 + 60*\a:1cm) $) -- %
 cycle;
}

\begin{scope}[yscale=1,xscale=-1]

\foreach \a in{0,2,4}{
\draw[fill=black] ($.111*(150+60*\a:2cm)+.111*(90+60*\a:4cm) + .333*(90 + 60*\a:2cm) $) circle (3.3333pt);
\draw[fill=black] ($.111*(150+60*\a:2cm)+.111*(90+60*\a:4cm) + .333*(90 + 60*\a:2cm)+.333*(90 + 60*\a:4cm) $) circle (3.3333pt);
\draw[fill=black] ($.111*(150+60*\a:2cm)+.111*(90+60*\a:4cm) + .333*(90 + 60*\a:4cm) +.333*(60*\a+110:2.6cm)$) circle (3.3333pt);

\draw[fill=black] ($.111*(150+60*\a:2cm)+.111*(90+60*\a:4cm) + .333*(150 + 60*\a:1cm) $) circle (3.3333pt);
\draw[fill=black] ($.111*(150+60*\a:2cm)+.111*(90+60*\a:4cm) + .333*(150 + 60*\a:2cm) + .333*(150 + 60*\a:1cm) $) circle (3.3333pt);
\draw[fill=black] ($.111*(150+60*\a:2cm)+.111*(90+60*\a:4cm) + .333*(150 + 60*\a:2cm) +.333*(60*\a+110:2.6cm)$) circle (3.3333pt);
}
\foreach \a in{0,2,4}{
\draw [very thick]
($.111*(150+60*\a:2cm)+.111*(90+60*\a:4cm) + .333*(90 + 60*\a:2cm) $)--
 ($.111*(150+60*\a:2cm)+.111*(90+60*\a:4cm) + .333*(90 + 60*\a:2cm)+.333*(90 + 60*\a:4cm) $) --
 ($.111*(150+60*\a:2cm)+.111*(90+60*\a:4cm) + .333*(90 + 60*\a:4cm) +.333*(60*\a+110:2.6cm)$)  --
($.111*(150+60*\a:2cm)+.111*(90+60*\a:4cm) + .333*(150 + 60*\a:2cm) +.333*(60*\a+110:2.6cm)$) --
 ($.111*(150+60*\a:2cm)+.111*(90+60*\a:4cm) + .333*(150 + 60*\a:2cm) + .333*(150 + 60*\a:1cm) $) --
 ($.111*(150+60*\a:2cm)+.111*(90+60*\a:4cm) + .333*(150 + 60*\a:1cm) $) -- cycle;
}

\end{scope}

\end{tikzpicture}\caption{Adjacency (dual) graph $G_2$, in bold,
and the barycentric subdivision graph
pictured together with the thin image of $G_2^T$.
}\label{barycentricsubdivision2}
\end{figure}
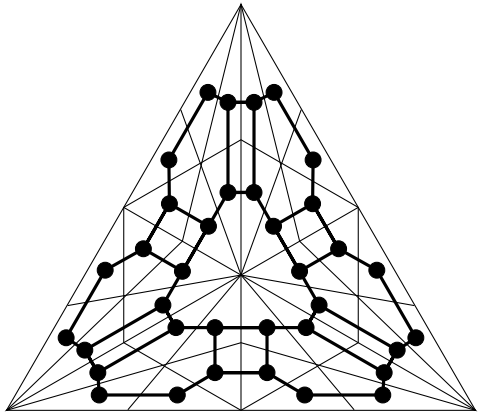

\begin{figure}[htb]
\includegraphics[height=2in]{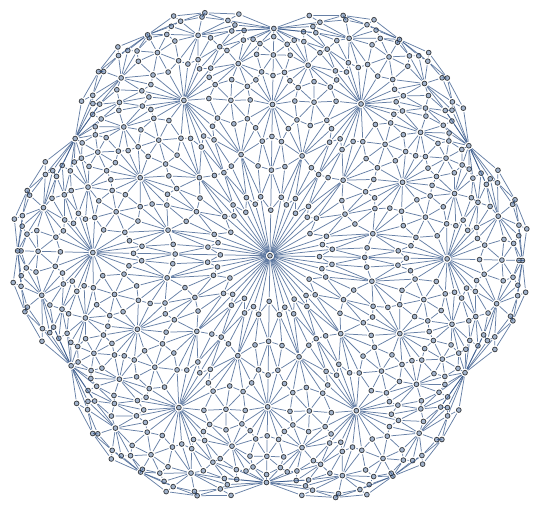}
\includegraphics[height=2in]{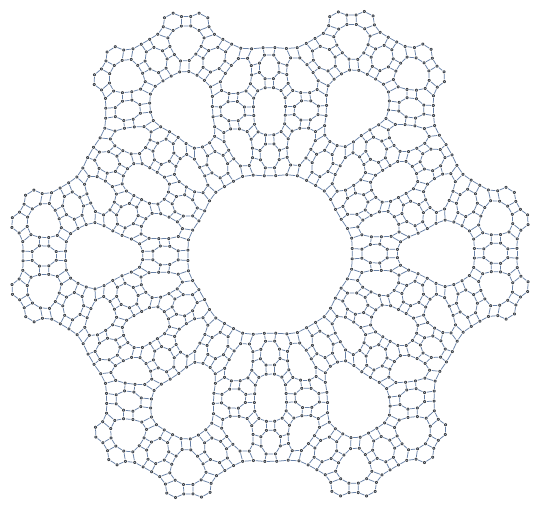}
\caption{
On the left: the graph  $G_4^T$ for barycentric subdivision of a 2-simplex.
On the right: the adjacency (dual) graph $G_4$.}
\label{Gn}
\end{figure}

The repeated barycentric subdivision of a simplex is a classical and fundamental notion from algebraic topology, see \cite[and references therein]{Hatcher}.
Recently it was considered from a probabilistic point of view in \cite{DF99,DM11,DMcM,DMic,Volkov} and graph theory point of view in \cite{Knill1,Knill2}.
Understanding how resistance scales on finite approximating graphs is the first step to developing analysis on fractals and fractal-like structures, such as self-similar graphs and groups, see \cite[and references therein]{BGN,GN,Kai05,NT08,Nek05}.
For finitely ramified post-critically finite fractals,
including nested fractals,
the resistance scales by the same factor between any two levels of approximating graphs
(see \cite[and references therein]{Kig93,Kig01,Lind}),
and this fact can be used  to prove the existence and uniqueness of a Dirichlet form on the limiting fractal structures.
In the infinitely ramified case,  resistance estimates are more difficult to obtain,
but are just as important to understanding diffusions on fractals.
Barlow and Bass \cite{Barlow,BB90,BBS90,BB99,BBKT,Bass} proved such estimates for the Sierpinski carpet and its generalizations. These techniques were extended to understanding resistance estimates between more complicated regions of the Sierpinski carpet, see \cite{McG02}.
The paper \cite{KZ92} provides another technique for proving the existence of Dirichlet forms  on non-finitely ramified self-similar fractals, which estimates the parameter $\rho$ by studying the Poincar\'e inequalities on the  approximating graphs of the fractals.
The long term motivation for our work comes from
probability and analysis on fractals
\cite{Barlow98,RT,Str1,Str2,Str3,Strbook},
vector analysis for Dirichlet forms
\cite{HT12,HKT12,HKT13,HRT,HT15,IRT12,LS14},
and especially from the works on the heat kernel estimates
\cite{BB90,BB99,BBKT,GH,GT,Kajino,Kajino13,Ki1,Ki2,Ki3,Li86,LY86,Telcs}.

In general terms, a Dirichlet form on a fractal is a bilinear form which is analogous to the classic Dirichlet energy on $\R^d$ given by
$
\eng(f) = \int |\nabla f|^2 \ dx$.
Dirichlet forms have many applications in geometry, analysis and probability. The theory of Dirichlet forms is equivalent, in a certain sense, to the theory of symmetric Markov processes,
see \cite{BH,ChenF,FOT11}.
The potential theoretic properties of the Dirichlet form have implications for this
stochastic  process. In particular, the resistance between two boundary sets is related to   the crossing times.
In the discrete setting, the Dirichlet form is the graph energy. In this case the resistance between two sets is determined  using Kirchhoff's laws. 
 For a more thorough introduction to these topics, one can see, for example, \cite{DS84,Lyo14}.

Although the results of Barlow and Bass et al are applicable to a large class of fractals,
we concentrate on one prototypical but difficult  to analyze    generalization of the classical  Sierpinski carpet.
Our work further develops existing techniques to
obtain resistance scaling estimates for the 1-skeleton of $n$-times iterated barycentric subdivisions of a triangle which we will denote $G^T_n$, and its weak dual the hexacarpet (introduced in \cite{BKN+12}), which we will denote $G_n$.  In our case, on the 1-skeleton $G_n^T$, the Markov process jumps between corners of the triangles in the subdivision. Our theoretical estimates correspond to a process with limiting spectral dimension between 2.28 and 2.38.
The Markov process on the hexacarpet graphs (which are denoted by $G$ and $G^H$ later on) corresponds to a random walk which jumps between the centers of these triangles, with spectral dimension between 1.63 and 1.77 ($\approx 1.74$ using the  numerical estimates in \cite{BKN+12}). This is a substantial difference implying, in particular, that one Markov process
is not
recurrent, while the other is
recurrent.
From the point of view of fractal analysis, our results suggest that
the   corresponding self-similar diffusion is not unique,
unlike \cite{HMT06,BBKT}.

If  $R_n^T$ and $R_n$ is the resistance between the appropriate boundaries in $G_n^T$ and $G_n^H$ respectively (see Figures~\ref{barycentricsubdivision1}, \ref{barycentricsubdivision2},
\ref{Gn}, \ref{boundaryproblem}), then we prove that the resistance $R_n^T$ and $R_n$ scale by constants $\rho^T$ and $\rho$ respectively, and obtain estimates on these constants. Note that in the current work the hexacarpet graph  $G^H_n$ is a modification of $G_n$ by adding a set of ``boundary'' vertexes. Our main result is the following theorem.

\begin{thm} \label{mainthm}
The resistances across graphs
$G_n^T$ and $G_n^H$
(defined in Subsection~\ref{sec-problem})
are reciprocals, that is
$
R^T_n = 1/R_n
$,
and  the asymptotic limits
\[
\log\rho^T = \lim_{n\to\infty} \frac1n \log R_n^T
\quad\text{\ \ and\ \ }\quad
\log\rho = \lim_{n\to\infty} \frac1n \log R_n
\]
exist (and $\rho^T=1/\rho$). Furthermore,
$
2/3\leq \rho^T   \leq 4/5$ and
$
5/4 \leq \rho \leq 3/2$.
\end{thm}

These estimates agree with the numerical experiments from \cite{BKN+12}, which suggest that there exists a limiting Dirichlet form on these fractals and estimates $\rho \approx 1.306$, and hence $\rho^T \approx 0.7655$. 

\begin{conjecture}
In the case $
5/4 \leq \rho \leq 3/2$ ($\rho \approx 1.306$), we conjecture that the recent results of A.~Grigor'yan, J.~Hu, K.-S.~Lau and M.~Yang in \cite{GH,GHL1,GHL2,GY2018} can imply existence of the Dirichlet form. 
\end{conjecture}

\begin{conjecture}
Since $2/3\leq \rho^T   \leq 4/5<5/4 \leq \rho \leq 3/2 $, we conjecture that there is essentially no uniqueness of the Dirichlet forms, spectral dimensions, resistance scaling factors etc   
for repeated barycentric subdivisions.
\end{conjecture}

This paper is organized as follows. Subsection~\ref{sec-appendix} defines general graph energy.  Subsection~\ref{sec-problem} lays out the   definitions of $G_n^H$ and $G_n^T$
and shows how to take advantage of the duality    to prove that $R_n = 1/R^T_n$. In Section~\ref{sec-upper} we   prove sub-multiplicative estimates $R_{m+n}\leq c R_mR_n$ for some constant $c$ independent of $m$ and $n$ in a fashion generalized from \cite{BB90}.  Then Fekete's theorem implies that the limits $\rho$ and $\rho^T$ exist.
 To show that these limits are finite, in
  Subsections~\ref{sec-short} and \ref{sec-short2}
 we prove upper and lower estimates on the constants $\rho$ and $\rho^T$
 by establishing upper and lower estimates on $R_n^T$ and $R_n$.
This is done by comparing $G_n^T$ and $G_n^H$ to subgraphs and quotient graphs respectively.


\section{Energy, resistance and duality}
\subsection{Energy, potentials, and flows on graphs} \label{sec-appendix}
This subsection collects the basic definitions and facts about  graph energies and resistances, which will be used in the current work. For  more detailed expositions on this subject see, for instance, \cite{BB90,DS84,Lyo14}. Let $G = (V,E)$ be a finite graph with vertex set $V$, and edge set $E$, which is a symmetric subset of $V\times V$. We define the set $\ell(V) = \set{f:V\to \R}$ to be the set of real valued functions on $V$, which we will sometimes refer to as functions or potentials on the graph $G$.
For all $p,q\in V$ we define conductances (weights) $c_{p,q} = c_{q,p}$  such that $c_{p,q} >0$ if $(p,q)\in E$ and $c_{p,q} = 0$ when $(p,q)\notin E$. Resistances between points $p$ and $q$ for any $(p,q)\in E$ will be defined $r_{p,q}:= 1/c_{p,q}$. For a graph with associated conductances, we define the graph energy
\[
\eng: \ell(V)\times \ell(V) \to \R,\quad\text{ \
} \quad \eng(f,g) = \frac12 \sum_{(p,q) \in E} c_{p,q}(f(p)-f(q))(g(p)-g(q)),
\]
and write $\eng(f):=\eng(f,f)$.

Antisymmetric   functions (with orientation) on the edge set will be denoted
\[
\ell^a(E) = \set{J:V\times V\to \R~|~J(p,q) = -J(q,p),\text{ and }J(p,q)=0\text{ if } (p,q)\notin E},
\]
and we define the energy dissipation $\operatorname{E}: \ell^a(E) \times \ell^a(E) \to \R$ to be the inner  product
\[
\operatorname{E}(J,K) = \frac12 \sum_{(p,q) \in E} r_{p,q} J(p,q)K(p,q)
\]
on $\ell^a(E)$. The discrete gradient $\nabla: \ell(V)\to\ell^a(E)$ is given by $$\nabla f (p,q) = c_{p,q}(f(p)-f(q)),$$ the discrete divergence $\div:\ell^a(E)\to\ell(V)$ is defined by $$\div J (p) = -\sum_{q \ : \ (p,q)\in E} J(p,q),$$ 
and the  associated   Laplace operator $\Delta: \ell(V)\to\ell(V)$ is defined by
\[
-\Delta f(p)  = -\div \nabla f(p) = \sum_{q:(p,q)\in E} c_{p,q}(f(q)-f(p)).
\]
  We follow the usual probabilistically and physically inspired convention where 
  $\Delta$ is a non-positive operator. This is analogous to the classical second derivative Laplace operator
  $\Delta=\frac {d^2}{dx^2}$ on $\mathbb R^1$, which is also
  non-positive, and sometimes is also denoted as ``$\nabla^2$''.
    In the more combinatorially and algebraically oriented literature
  $L=-\Delta$ is sometimes called the weighted graph Laplacian.

For $A,B\subset V$, $J\in \ell^a(E)$ is called a flow from $A$ to $B$, if $\div J (p) = 0$ for $p\notin A\cup B$. The flux of a flow $J$ from $A$ to $B$ is defined by
\[
\flux(A,B,J) = \sum_{p\in A}\div J(p) = - \sum_{p\in B} \div J(p).
\]
The effective resistance between sets $A$ and $B$ is defined by
\[
R(A,B)^{-1} = \inf\set{\eng(f)~|~f|_A \equiv 0,~f|_B\equiv 1} .
\]
Energy is minimized by the function $\phi$ such that $\phi|_A \equiv 0$, $\phi|_B\equiv 1$ and $\Delta \phi(p) = 0$ for $p\notin A\cup B$, and thus $\eng(\phi) = 1/R(A,B)$. We  refer to such a function $\phi$ as the harmonic function with boundary $A$ and $B$. The only function which satisfies $\Delta f \equiv 0$ (i.e. harmonic without boundary) and $f|_{A\cup B}\equiv 0$ is the constant $0$ function, thus the $\phi$ is unique. Similarly $I = R(A,B)^{-1} \nabla \phi$ is the unique energy minimizing flow from $A$ to $B$ with $\operatorname{E}(I) = 1$. The following four characterizations of $R(A,B)$ which are equivalent to the original, as seen in Section 2 of \cite{BB90}.
\begin{enumerate}
\item $R(A,B) = \sup\set{1/\eng(f)~:~f|_A \equiv0,~f|_B\equiv1}$
\item $R(A,B) = 1/\eng(\phi)$  where $\phi\in \ell(V)$ is the the unique function with $\phi|_A\equiv 0$,$\phi|_B \equiv 1$, $\Delta\phi(p) = 0$ for $p\notin A\cup B$.
\item $R(A,B) = \inf\set{\operatorname E(J)~:~\flux(A,B,J)=1}$.
\item $R(A,B) = \operatorname E(I)$ where $I\in\ell^a(E)$ is the unique energy dissipation minimizing unit flow from $A$ to $B$.
\end{enumerate}

\subsection{Barycentric subdivision, the hexacarpet and the resistance problem}\label{sec-problem}


We define the $2$-simplicial complexes $\T_n = (\E_n^0,\E^1_n,\E^2_n)$ where $\E_n^i$ are the $i$-simplexes of the complex,  starting with a $2$-simplex (a triangle) $\T_0$ with $0$-simplexes (vertexes) $\E_0^0 = \set{p_0,p_1,p_2}$, $1$-simplexes (edges) $\E_0^1 =\set{[p_i,p_j]}_{i\neq j}$, and $2$-simplex (triangle) $\E_0^2 = \set{[p_0,p_1,p_2]}$. Here, if $q_0,q_1,q_2\in\E_n^0$, then $[q_0,q_1]$ will refer to the $1$-simplex with $q_0$ and $q_1$ as endpoints (which may or may not be in $\E_n^1$), and similarly for $[q_0,q_1,q_2]$. We will only be considering simple simplicial complexes (without multiple edges/triangles).
Thus $[q_0,q_1]=[q_1,q_0]$ determines a unique 1-simplex for example. We  also use the notation $<$ to denote containment of simplexes, i.e. $q_0<[q_0,q_1]<[q_0,q_1,q_2]$.
$\T_{n+1}$ is defined inductively from $\T_n$ by barycentric subdivision, pictured in Figure~\ref{barycentricsubdivision1}. That is $\E_{n+1}^0$ is $\E_{n}^0$ along with the barycenters of simplexes in $\E_{n}^i$, $i=1,2$, which we  refer to as $\mathsf{b}(e)$ for $e\in \E_{n}^i$, $i=1,2$. $1$- and $2$-simplexes of $\T_{n+1}$ are formed from barycenters of nested simplexes. More concretely put, elements of $\E_{n+1}^1$ are either of the form $ [q_0, \mathsf{b}([q_0,q_1])]$ where $q_0,q_1\in \E_n^0$ with $[q_0,q_1]\in\E_n^1$, $[q_0,\mathsf{b}([q_0,q_1,q_2])]$, or $[\mathsf b([q_0,q_1]),\mathsf b([q_0,q_1,q_2])]$ where $q_0,q_1,q_2\in \E_n^0$ with  $[q_0,q_1,q_2]\in\E_n^2$, and elements of $\E_n^2$ are of the form $[q_0,\mathsf{b}([q_0,q_1]),\mathsf{b}([q_0,q_1,q_2])]$ where $q_0,q_1,q_2\in \E_n^0$ with $[q_0,q_1]\in\E_n^1$ and $[q_0,q_1,q_2]\in\E_n^2$.

\begin{defn}
We define the graph $G^T_n = (V_n^T,E_n^T)$ with vertex set $V_n^T = \E_n^0$ and edge relation $q\sim_T q'$ if $[q,q']\in \E_n^1$. We will refer to this as the 1-skeleton of the $\T_n$.
\end{defn}

\begin{defn}
Following \cite{BKN+12}, we define the graph $G_n = (V_n,E_n)$ with vertex set $V_n = \E_n^2$ and edge relation $[q_0,q_1,q]\sim[q_0,q_1,q']$. That is, the vertexes of $G_n$ are the $2$-simplexes of $\T_n$ and they are connected by an edge if these simplexes share a $1$-simplex.
\end{defn}

\begin{remark}
$G^T_n$ is classically known to be a planar graph, as seen in Figure~\ref{barycentricsubdivision1}, although throughout this work we will refer to the hexagonal embedding from Figure~\ref{boundaryproblem} more often. With either of these embeddings, $G_n$ is the weak planar dual, that is each of its vertexes correspond to a plane region carved out by the embedding of $G^T_n$ with the exception of the unbounded component.
\end{remark}

We will need explicit names for the elements of $\E_1^0 = \set{p_0,p_1,p_2,p'_0,p'_1,p_2', p'}$ where $\set{p_0,p_1,p_2} = \E^0_0$ as above, $p'_i = \mathsf{b}([p_i,p_{i'}])$, $i'\equiv i+1\mod 3$, and $p' = \mathsf{b}([p_0,p_1,p_2])$.

This is convenient for recursively defining functions on $\T_n$. Define self-similarity maps $F_i:\T_n\to\T_{n+1}$ for $i=0,1,\ldots,5$, which are defined on $\T_0$ by $F_i(p_0) = p'$, $F_i(p_1) = p_{\lceil i/2\rceil }$, and $F_i(p_2) = p_{ \lfloor i/2\rfloor }'$, where the index is taken mod 3. $F_i$ is extended to $\T_n$ by the relations  $F_i([q_0,q_1]) = [F_i(q_0),F_i(q_1)]$, $F_i([q_0,q_1,q_2]) = [F_i(q_0),F_i(q_1),F_i(q_2)]$ and $\mathsf b \circ F_i = F_i \circ \mathsf b$. If $w = w_1w_2\cdots w_k$ is a word in $\set{0,1,\ldots,5}^k$, then we define $F_w: \T_n^{T} \to \T_{n+k}^{T}$ by $F_w := F_{w_1}\circ F_{w_2}\circ\cdots\circ F_{w_k}$.
$F_i$ as a function from $G^T_n \to G^T_{n+1}$ or $G_n\to G_{n+1}$ is a  graph homomorphism, and $G^{T}_{n+1} = \cup_{i=0}^5 F_i(G_n^{T})$. However, this is not true for $G_n$ and $G_{n+1}$, since not every edge of $G_{n+1}$ is covered by $F_i(G_n)$ for some $i=0,1,\ldots,5$. We want to take advantage of this self-similarity throughout the current work, thus we  define a modified hexacarpet graph.

\begin{defn}
The (modified) hexacarpet graph $G^H_n$ is defined to have vertex set
\[
V_n^H :=\E_n^2 \cup \E_n^1
\]
where adjacency is determined by $[q_0,q_1,q]\sim_H[q_0,q_1]$, i.e., $e\in \E_n^1$ is connected to $f\in \E_n^2$ if $e <f$.
\end{defn}

For $G^T_n$, define the conductance of edges
\[
c^T_{q,q'}=\begin{cases}
0 & \text{if }[q,q']\notin\E_n^1 \\
1 & \text{if there exists } q_0 \neq q_0' \text{ with } [q,q',q_0],[q,q',q_0'] \in \E_n^2 \\
2 & \text{if there exist only one }q_0\text{ such that }[q,q',q_0]\in\E_n^2.
\end{cases}
\]
We take $
\eng_n^T
$
to be the graph energy defined with the above conductance. The advantage of  these conductance values is the resulting self-similarity relation
\[
\eng_{n+1}^T (f) = \sum_{i=0}^5 \eng_n^T (f\circ F_i).
\]
Similarly if we define $c_{q,q'}^H = 1/2$ if $q\sim_H q'$ and $0$ otherwise, then the resulting energy function $\eng^H_n$ satisfies the following relation
\[
\eng_{n+1}^H (f) = \sum_{i=0}^5 \eng_n^H (f\circ F_i).
\]
Both of these relations are also true for energy dissipation of functions on edges of these graphs.
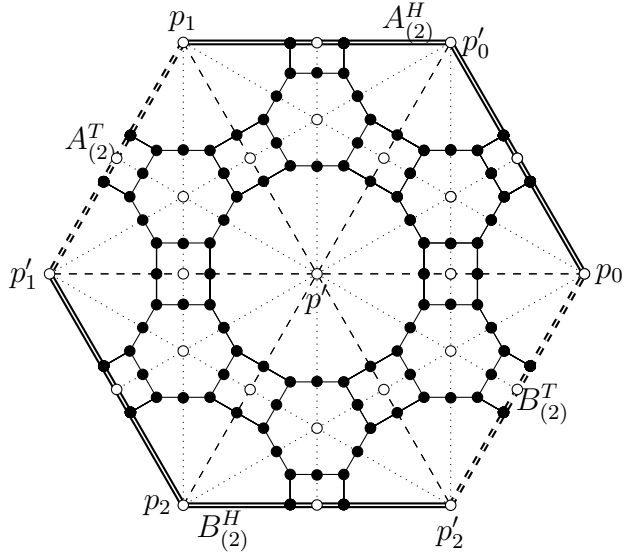
\begin{figure}
\inviz{
\begin{tikzpicture}[scale=.8888]

\foreach \a in {0,1,2,...,5}{
\coordinate (\a) at ($(60*\a:4)$);

\draw[dashed] ($(60*\a:4)$)--($(60+60*\a:4)$);

\draw[dashed] ($(120*\a:4)$)--($(-180+120*\a:4)$);

\draw[dotted] (0:0)--($.5*(60*\a:4)+.5*(-60+60*\a:4)$);
\draw[dotted]($(60*\a:4)$)--($(-60+60*\a:2)$);
\draw[dotted] ($(-120+60*\a:2)$)--($(180+60*\a:4)$);
}

\draw[thick,double] (-180:4)node[left]{$p_1'$}--(-120:4) node[left] {$p_2$}--(-60:4) node[below] {$p_2'$};
\draw[thick,double,dashed] (-60:4)--(0:4);
\draw[thick,double] (0:4)node[right]{$p_0$}--(60:4)node[right] {$p_0'$}--(120:4)node[above] {$p_1$};
\draw[thick,double,dashed] (120:4)--(180:4);
\node at (150:3.90) {$A^T_{(2)}$};
\node at (-30:3.90) {$B^T_{(2)}$};
\node at (70:4) {$A^H_{(2)}$};
\node at (250:4.1) {$B^H_{(2)}$};

\foreach \a in {0,1,...,6}{
\foreach \b in {0,1,2,...,6}{

\draw ($(30+60*\a:2.31)+(60*\b:.8)$)--($(30+60*\a:2.31)+(60+60*\b:.8)$);
\draw ($(30+60*\a:2.31)+(180+60*\a:.8)$)--($(90+60*\a:2.31)+(-60+60*\a:.8)$);
\draw ($(30+60*\a:2.31)+(120+60*\a:.8)$)--($(90+60*\a:2.31)+(60*\a:.8)$);
\draw ($(30+60*\a:2.31)+(60*\a:.8)$)--($(30+60*\a:2.31)+(60*\a:.8)+(30+60*\a:.45)$);
\draw ($(30+60*\a:2.31)+(60+60*\a:.8)$)--($(30+60*\a:2.31)+(60+60*\a:.8)+(30+60*\a:.45)$);
}}

\foreach \a in {0,1,...,6}{
\foreach \b in {0,1,2,...,6}{
\draw[fill=black] ($(30+60*\a:2.31)+(60*\b:.8)$)                     circle (2.2222pt);
\draw[fill=black] ($(30+60*\a:2.31)+(60*\a:.8)+(30+60*\a:.45)$)      circle (2.2222pt);
\draw[fill=black] ($(30+60*\a:2.31)+(60+60*\a:.8)+(30+60*\a:.45)$)   circle (2.2222pt);
\draw[fill=black] ($(30+60*\a:2.31)+(30+60*\b:.7)$)                     circle (2.2222pt);
}

\draw[fill=black] ($(60*\a:2.4)$) circle (2.2222pt);
\draw[fill=black] ($(60*\a:1.6)$) circle (2.2222pt);
\draw[fill=white] ($(60*\a:4)$) circle (2.2222pt);
\draw[fill=white] ($(60*\a:2)$)circle (2.2222pt);
\draw[fill=white] ($(30+60*\a:3.46)$)circle (2.2222pt);
\draw[fill=white] ($(30+60*\a:2.31)$)circle (2.2222pt);
}
\draw[fill=white] (0,0) circle (2.2222pt) node[below] {$p'$};

\end{tikzpicture}
}
\caption{$A_{(2)}^{T/H}$ and $B_{(2)}^{T/H}$ on the hexagonal embedding of $G_2^{T/H}$.}\label{boundaryproblem}
\end{figure}

It will often be useful to think of these graphs as embedded in the plane $\R^2$. It is typical to think of $\T_n$ as a subdivided triangle in the plane, but we  embed it as a hexagon, as $\T_n$, $n>0$, has symmetry group $D_6$, the dihedral group on $6$ elements. As such, for $n> 0$ we define a map $F^T: \E_n^0 = V_n^T \to \R^2$ by $F^T(p') = (0,0)$,
\[
F^T(p_k) = (\cos(2k\pi/3), \sin(2k\pi/3)), \quad  F^T(p_k') = (\cos((2k+1)\pi/6),\sin((2k+1)\pi/6))
\]
for $k=0,1,2$. Thus $\E_1^0$ is mapped to the corners and midpoint of a regular hexagon centered at $(0,0)$, see Figure~\ref{boundaryproblem}. We extend to $\E_n^0$ by taking averages:
\[
F^T\circ\mathsf b([q_0,q_1]) = \frac{F^T(q_0){+}F^T(q_1)}{2},\quad   F^T\circ\mathsf b([q_0,q_1,q_2]) = \frac{F^T(q_0){+}F^T(q_1){+}F^T(q_2)}{3}.
\]
We embed $G^H_n$ by the map $F^H:V_n^H \to \R^2$ by $F^H = F^T\circ\mathsf b$. Thus the vertexes of the embedded $G^H_n$ are the centers of the triangles and edges of the embedding of $G^T_n$.

If $e = [q_0,q_1] \in \E_n^1$ then we define the geometric realization of $e$ , $|e|$, to be the convex hull of $F^T(p_0)$ and $F^T(p_1)$. That is
\[
|e| = \set{\theta_0 F^T(p_0) + \theta_1F^T(p_1)~:~ \theta_0+\theta_1 = 1}.
\]
Similarly, if $f=[q_0,q_1,q_2] \in \E_n^2$, then $|f|$ is defined to be the convex hull of $F^T(p_0)$, $F^T(p_1)$ and $F^T(p_2)$.

To define the resistance problem on these graphs, we   need to define the boundary of these graphs.   $|e| \subset |f|$, for $e\in E_k^i$ and $f\in\E_n^i$
if the geometric realization of $e$ is a subset of the geometric realization of $f$, e.g. $|\mathsf b(e)|\subset |e|$ or $|[q_0,\mathsf b ([q_0,q_1])]| \subset |[q_0,q_1,q_2]|$.
We define $L^{(n,T)}_i$, resp. $L^{(n,H)}_i$, $i=0,1,\ldots 5$ to be the vertexes  $q\in V_n^T$, resp. $V_n^H$, such that $|q|\subset |[p_{ i /2},p_{i/2}']|$  if $i$ is even, and the set of $|q|\subset |[p_{ j},p_{(i-1)/2}']|$ where $ j \equiv (i+1)/2 \mod 3$  if $i$ is odd. We will suppress arguments of the superscript when there is no danger of confusion.

\begin{defn}
Define $A_{(n)}^H = L_0^{(n,H)}\cup L_1^{(n,H)}$ and $B_{(n)}^H = L_3^{(n,H)}\cup L_4^{(n,H)}$. Further, define $R_n$ to be the effective resistance with respect to $\eng^H_n$ between $A_{(n)}^H$ and $B_{(n)}^H$.
\end{defn}

\begin{defn}
Define $A^T_{(n)} = L_2^{(n,T)}$ and $B^T_{(n)}= L_5^{(n,T)}$. Further, define $R_n^T$ to be the effective resistance with respect to $\eng^T_n$ between $A_{(n)}^T$ and $B_{(n)}^T$.
\end{defn}

\begin{theorem}
The resistances $R_n$ and $R_n^T$ are related by
$
R^T_n = 1/{R_n}.
$
\end{theorem}

\begin{figure}
\inviz{
\begin{tikzpicture}[scale=.8888]

\foreach \a in {0,1,2,...,5}{
\coordinate (\a) at ($(60*\a:4)$);

\draw[dashed] ($(60*\a:4)$)--($(60+60*\a:4)$);

\draw[dashed] ($(120*\a:4)$)--($(-180+120*\a:4)$);

\draw[dashed] (0:0)--($.5*(60*\a:4)+.5*(-60+60*\a:4)$);
\draw[dashed]($(60*\a:4)$)--($(-60+60*\a:2)$);
\draw[dashed] ($(-120+60*\a:2)$)--($(180+60*\a:4)$);
}

\draw[fill=white] (0,0) circle (2.2222pt);
\draw[thick,double] (-37:3.8)--(0:4.35) node[right] {$b^H$}--(37:3.8);
\draw[thick,double] (143:3.8)--(180:4.35) node[left] {$a^H$}--(217:3.8);
\draw[thick] (60:4.1)--node[above] {$a^T$}(120:4.1) ;
\draw[thick] (60:3.9)--(120:3.9) ;
\draw[thick] (-60:4.1)--node[below] {$b^T$}(-120:4.1);
\draw[thick] (-60:3.9)--(-120:3.9) ;

\foreach \a in {0,1,...,6}{
\foreach \b in {0,1,2,...,6}{
\draw ($(30+60*\a:2.31)+(60*\b:.8)$)--($(30+60*\a:2.31)+(60+60*\b:.8)$);
\draw ($(30+60*\a:2.31)+(180+60*\a:.8)$)--($(90+60*\a:2.31)+(-60+60*\a:.8)$);
\draw ($(30+60*\a:2.31)+(120+60*\a:.8)$)--($(90+60*\a:2.31)+(60*\a:.8)$);
\draw[fill=black] ($(30+60*\a:2.31)+(60*\b:.8)$)                     circle (2.2222pt);
\draw[fill=white] ($(30+60*\a:2.31)$) circle (2.2222pt);
\draw[fill=white] ($(60*\a:2)$) circle (2.2222pt);
\draw[fill=white] ($(30+60*\a:3.47)$) circle (2.2222pt);
\draw[fill=white] ($(60*\a:4)$) circle (2.2222pt);
}}
\foreach \a in {0,2,3,5,6}{
\draw[fill=black] ($(30+60*\a:2.31)+(60*\a:.8)+(30+60*\a:.75)$)      circle (2.2222pt);
\draw[fill=black] ($(30+60*\a:2.31)+(60+60*\a:.8)+(30+60*\a:.75)$)   circle (2.2222pt);
\draw ($(30+60*\a:2.31)+(60*\a:.8)$)--($(30+60*\a:2.31)+(60*\a:.8)+(30+60*\a:.75)$);
\draw ($(30+60*\a:2.31)+(60+60*\a:.8)$)--($(30+60*\a:2.31)+(60+60*\a:.8)+(30+60*\a:.75)$);

}

\end{tikzpicture}
}
\caption{$\tilde G^H_2$ and $\tilde G_2^T$ without the additional edges. 
}\label{proofoduality}
\end{figure}
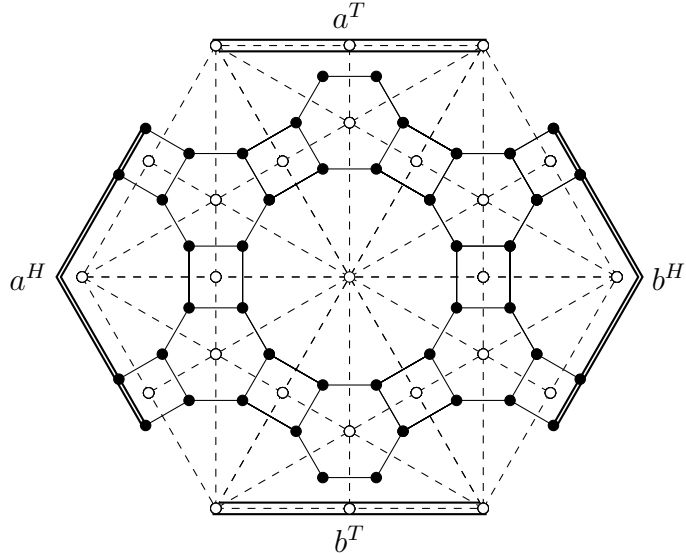

\begin{proof}
The main tool of this proof is proposition 9.4 from \cite{Lyo14}. 
Define the weighted graph $\widetilde{ G}^T_n = (\widetilde{V}^T_n, \widetilde{E}_n^T)$ 
where $\widetilde{V}^T_n$ is $V^T_n$ modulo the relation which identifies all elements of $A^T_{(n)}$ and $B^T_{(n)}$ into two vertexes called $a^T$ and $b^T$ respectively. The 
effective resistance between $a^T$ and $b^T$ with respect to $\tilde\eng^T_n$ is $R^T_n$.

Further define $\widetilde G^H_n$ by, not only identifying $A^H_{(n)}$ and $B^H_{(n)}$ into single vertexes $a^H$ and $b^H$ respectively, but also to replace the sequential edges of the form $[q_0,q_1,q]\sim_H[q_0,q_1]\sim_H[q_0,q_1,q']$ using Kirchoff's laws. Thus the sequential connections with resistance $1/2$  are replaced with one connection $[q_0,q_1,q]\sim_H[q_0,q_1,q']$ with resistance $1$. It is easy to see that the resistance between $a^H$ and $b^H$ is $R_n$. Also, remove the vertexes contained in $A^T_{(n)}$ and $B^T_{(n)}$ and associated edges --- since these vertexes are connected to the graph by only 1 edge,removing them has no impact on the resistance.

If we define the graphs $(\tilde G^H_n)^\dagger$ and $(\tilde G^T_n)^\dagger$ to be the $\tilde G^H_n$ and $\tilde G^T_n$ with an additional edge connecting $a^H$ to $b^H$ and $a^T$ to $b^T$ respectively, then $(\tilde G^H_n)^\dagger$ is the planar dual of $(\tilde G^T_n)^\dagger$  (see Figure~\ref{proofoduality}), and thus by Proposition 9.4 in \cite{Lyo14}, $R_n = 1/R^T_n$.
\end{proof}

\section{Multiplicative estimates and other proofs}\label{sec-upper}


\begin{theorem}\label{thm-multiplicative}
$R_{m+n} \leq \frac43 R_n R_m$, and equivalently $R^T_{m+n}\geq \frac34 R^T_mR^T_n$.
\end{theorem}
We give two proofs of this theorem.
The constants differ in the proofs, but the exact value of the constant does not affect the existence of $\rho$ and $\rho^T$. This establishes the result independently of the duality of the graphs
(variations on these proofs work for different choices of boundary).
One version proves the upper estimate on $R_{m+n}$ directly, and uses flows on $G_n^H$. The direct proof of the lower estimate on $R^T_{m+n}$ is proven using potentials on $G^T_n$. The two proofs mirror the upper and lower bounds for the resistance of the pre-carpet approximations for the Sierpinski carpet in \cite{BB90}. The two versions of our proof highlight the importance of duality in proving Barlow--Bass style resistance estimates, and suggest possible
generalizations.

Theorem~\ref{thm-multiplicative} implies that there are constants  $c_0$ and $c_1$ such that $c_0 +\log R_n^T$ and $c_1+\log R_n$ are superadditive/subadditive positive sequences
and
thus
we have the following.
\begin{cor}\label{cor-rho}
The limits
$
\lim\limits_{n\to\infty}\frac1n\log R_n^T = \log\rho^T
$
 and
$
\lim\limits_{n\to\infty}\frac1n\log R_n = \log\rho
$
exist.
\end{cor}
Note that this corollary does not rule out the possibility that $\rho = 0$ or $\infty$. In Subsections~\ref{sec-short} and \ref{sec-short2}  we establish  positive upper and lower estimates on $\rho$ and $\rho^T$.

\subsection{Upper estimate and flows on $G_n^H$}

\begin{figure}
\includegraphics[width=.7777\textwidth]{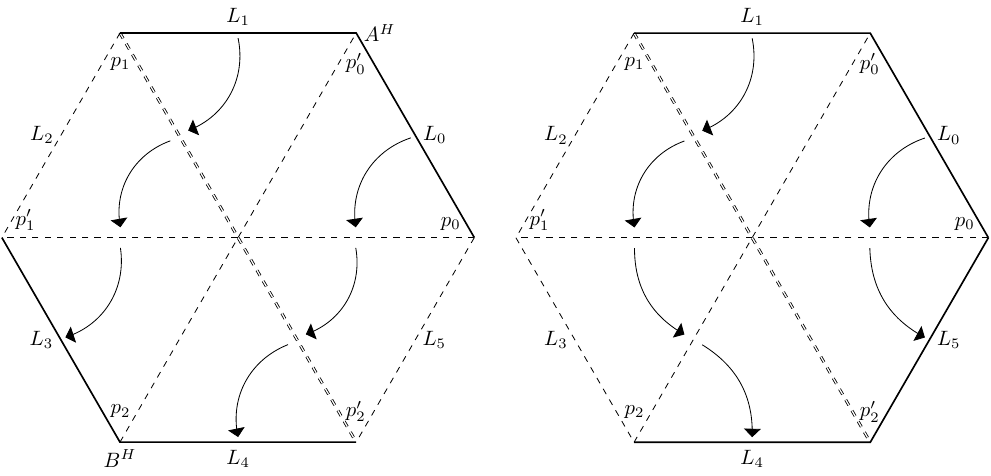}
\caption{\label{fig:flowH} The transformation from the flow $I^n$ (left) to the flow $H_{02}^n$ (right).}
\end{figure}

Consider the flow $I^n$ which is the minimizing flow on $G_n^H$ from $A^H_{(n)}$ to $B^H_{(n)}$. 
We will construct a flow on $G_{n+m}$ which has global structure resembling $I^m$ but on $F_\omega G_n^H$ for $\omega\in \set{0,1,\ldots,5}^m$ has structure which is built from $I^n$.

Define $H_{02}^n = I^n$ when restricted to the half of the $G_n^H$ which connects vertexes contained in $[p_0,p'_0,p']$, $[p_0',p_1,p']$, $[p_1,p'_1,p']$. On the other half, $H_{02}^n$ is $I^n$ composed with the symmetry which exchanges $p_0$ with $p_1'$, $p_0'$ with $p_1$, and $p_2$ with $p'_2$ and is extended to the rest of $V^H_{n}$ by convex combinations. The construction of $H_{02}^n$ is depicted in Figure~\ref{fig:flowH}. Using the embedding $F^H$, this symmetry is the reflection through the line which makes a $\pi/2$ angle with the $x$-axis. Because $I^n$ subjected to a $\pi$-rotation is $-I^n$, $H_{02}^n$ is a flow.

Thus $H_{02}^n$ is a flow on $G_n^H$ between $L_0^{(n)}\cup L^{(n)}_1$ and $L_4^{(n)}\cup L_5^{(n)}$, with flux $1$. 
 $H_{01}^n$ is the flow from $L_0^{(n)}\cup L^{(n)}_1$ and $L_2^{(n)}\cup L_3^{(n)}$ with flux $1$ obtained from $H_{02}^n$ by applying the appropriate symmetry. Since the values of $H^n_{01}$ (resp. $H_{02}^n$) are just a rearrangement of those in $I^n$, then $ E(H^n_{01})= E(H^n_{02})=E(I^n)=R_n$ for all $i\neq j$.

We   consider $G_m^H$ as the set of Y-networks, centered at $x\in \E_m^2 \subset V_m^H$. Take $a_0(x)$, $a_1(x)$, $a_2(x)$ to refer to the outward flow of $I^m$ restricted to each of these edges in the Y-network associated to $x$ with orientation such that $a_0(x)$ is of a different sign then $a_1(x)$ and $a_2(x)$. i.e. if $a_0(x)<0$, then $a_1(x),a_2(x)\geq 0$, if $a_0(x)>0$, $a_1(x),a_2(x)\leq 0$, and in the case when $a_0(x)=a_1(x)=a_2(x) = 0$ then the choice is arbitrary
and
\[
E(I^m) = \frac12 \sum_{x \in \E_m^2} (a_0(x)^2+a_1(x)^2+a_2(x)^2) = R_m.
\]

Take $F_\omega G_n^H\subset G_{n+m}^H$ such that $F_\omega ([p_0,p_1,p_2]) = x$  to be the subgraph of $G_{n+m}^H$ isomorphic to $G_n^H$ corresponding to $x\in \E_m^2$. We  assume that the labeling of the sides $F_\omega(L_2^n\cup L_3^n)$ is contained in the edge which corresponds to the values $a_1(x)$, and $F_\omega(L_4^n\cup L^n_5)$ corresponds to $a_2(x)$. The set of subgraphs $\set{F_\omega G_n^H}_{x\in \E_n^2,\omega\in \set{0,1,\ldots, 5}^m}$ cover $G_{m+n}^H$ and define the flow $J$ on $G_{m+n}^H$by its values on these subgraphs
\[
J\circ F_\omega = a_1(x)H^n_{01} + a_2(x)H^n_{02}.
\]
$J$ is a well defined flow because $H_{02}^n$ is obtained by a reflection of $H_{01}^n$ which is symmetric with respect to $I^n$, so $ aH_{01}^n(y,z) + bH_{02}^n(y,z) = (a+b)I^n(y,z)$ for all $y\in L_0^n\cup L_1^n$.

Now we see that
\begin{align*}
\operatorname E(J)& = \sum_{x\in \E^2_m} \operatorname E(a_1(x)H_{01}^n+a_2(x)H^n_{01})\\
 & = \sum_{x\in \E^2_m} \paren{a_1(x)^2  \operatorname E(H_{01}^n) + a_2(x)^2 \operatorname E(H_{02}^n) + 2a_1(x)a_2(x) \operatorname E(H_{01}^n,H_{02}^n)} \\
 & \leq \sum_{x\in \E^2_m} \paren{ a_1(x)^2\operatorname E(H_{01}^n) + a_2(x)^2\operatorname E(H_{02}^n) + 2a_1(x)a_2(x)\operatorname E(H_{01}^n)^{1/2}\operatorname E(H_{02}^n)^{1/2} }  \\
 & = \sum_{x\in \E^2_m} (a_1(x)+a_2(x))^2\operatorname E(I_n) = R_n\sum_{x\in \E^2_m} a_0(x)^2 \\
 & \leq R_n\frac23 \sum_{x \in \E^2_m} (a_0(x)^2+a_1(x)^2+a_2(x)^2) = \frac43 R_mR_n.
\end{align*}
Note that the first inequality holds because of Cauchy-Schwartz inequality and the fact that, by our labeling convention, $a_1(x)a_2(x)\geq 0$, and the last inequality holds because $ a_1^2(x) + a_2^2(x) \geq (a_1(x)+a_2(x))^2/2 = a_0^2(x)/2$.

\subsection{Lower estimate and potentials on $G^T_n$} \label{sec-lower}
\begin{figure}
\inviz{
\begin{tikzpicture}[scale=.7777]

\foreach \a in {0,1,2,...,5}{
\coordinate (\a) at ($(180+60*\a:4)$);

\draw ($(180+60*\a:4)$)--($(240+60*\a:4)$);

\draw ($(180+120*\a:4)$)--($(120*\a:4)$);

\draw[dotted] (90:0)--($.5*(180+60*\a:4)+.5*(120+60*\a:4)$);
\draw[dotted]($(180+60*\a:4)$)--($(120+60*\a:2)$);
\draw[dotted] ($(60*\a:2)$)--($(-60+60*\a:4)$);
}

\node at (210:2.3) {$\mathbb{1}-u$};
\node at (-90:2.25) {$\mathbb{1}-v$};
\node at (-30:2.25) {$v$};
\node at (30:2.25) {$u$};
\node at (90:2.25) {$w$};
\node at (150:2.3) {$\mathbb{1}-w$};

\draw[thick] (120:4)node[left] {$1/2$}--(180:4);
\draw[thick,double] (180:4) node[left] {$1$}--node[left] {$B^T$} (240:4)node[left] {$1$};
\draw[thick] (-60:4)node[right]{$1/2$}-- (0:4);
\draw[thick,double] (0:4)node[right] {$0$}--node[right] {$A^T$}(60:4)node[right] {$0$};

\end{tikzpicture}
}
\caption{ The function $u$, $v$ and $w$.}
\label{uvw}
\end{figure}
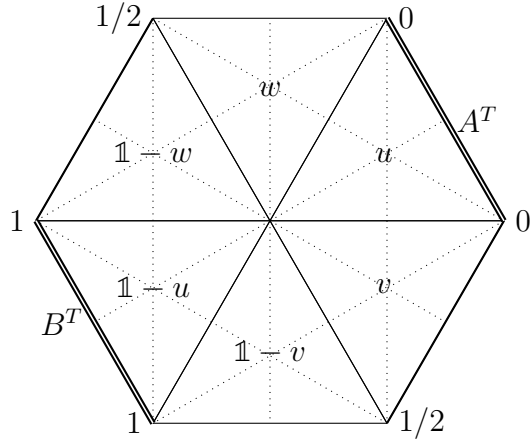

Let $\phi_n$ be the harmonic potential on $G^T_n$ with boundary values $0$ on $A^H_{(n)}$ and $1$ on $B^T_{(n)}$. On $G^T_{n-1}$, define $u = \phi_n \circ F_0$, $v= \phi_n\circ F_{1}$ and $w=\phi_n\circ F_5$. $\phi_n\circ F_2$, $\phi_n\circ F_3$, and $\phi_n\circ F_4$ can be written in terms of $u$, $v$, $w$ and the constant $\mathbb{1}$ function as illustrated in Figure~\ref{uvw}.

Notice that the function $w$ is $v\circ \sigma$ where $\sigma: G^T_{n-1}\to G^T_{n-1}$ is the symmetry which exchanges $p_0'$ and $p_2$, fixing $p_0$, and is extended to the rest of $G^T_{n-1}$ by averages (with respect to $F^T$, $\sigma$ is the flip about the horizontal axis). Also,
\[
R_n^{-1} = \eng_n^T(\phi_n) = 2(\eng_{n-1}^T(u)+\eng_{n-1}^T(v)+\eng_{n-1}^T(w)) = 2\eng_{n-1}^T(u)+4\eng_{n-1}^T(v).
\]

\begin{lemma}
$\eng_{n-1}^T(u,v-w) = 0$
\end{lemma}

\begin{proof}
On one hand, $\eng_{n-1}^T(f) = \eng_{n-1}^T(f\circ \sigma)$ for all $f$ because $\sigma$ is a graph isometry with respect to the conductances. However, the function $u$ symmetric about the horizontal axis, i.e. $u(x) = u\circ \sigma(x)$, and $v$ and $w$ is anti-symmetric about this axis, i.e.  $(v\circ\sigma(x)-w\circ\sigma(x)) =(w(x)-v(x))$. Thus
$
\eng_{n-1}^T(u,v-w) = \eng_{n-1}^T(u\circ\sigma,v\circ\sigma-w\circ\sigma)
= -\eng_{n-1}^T(u,v-w).
$
\end{proof}

For each $x \in \E^{2}_{n-1}$, $i' \equiv i+1 \mod 6$ define $a^x=\mathsf b(x)$ to be the barycenter of $x$. Also, define $a_i^x$, $i=0,1,\ldots,5$ to be vertexes and barycenters of edges contained in $x$ ordered in such at way that $[a_i^x,a_{i'}^x]$ is an edge in $\E^{1}_{n}$, and that the vectors
\begin{align*}
&(v\circ F_\omega(a^x),v\circ F_\omega(a^x_i),v\circ F_\omega(a^x_{i'})) = (1/2,0,1/2),\\
&(w\circ F_\omega(a^x),w\circ F_\omega(a^x_i),w\circ F_\omega(a^x_{i'})) = (1/2,1/2,0), \quad\text{and}\\
&(u\circ F_\omega(a^x),u\circ F_\omega(a^x_i),u\circ F_\omega(a^x_{i'})) = (1/2,0,0),
\end{align*}
for $\omega\in \set{0,1,\ldots,5}^{n-1}$ such that $F_\omega([p_0,p_1,p_2]) = x$.
Notice that for all $x\in \E^{2}_{n-1}$, $a^x$ and $a_i^x$ are contained in $\E^{0}_n$, and that, for any function $f $ on $G^T_n$, then
\[
\eng^T_n(f) = \sum_{x\in E^{(2)}_{n-1}}
\sum_{i=0}^{5}\paren{ \paren{f(a^x)-f(a_i^x)}^2 - \frac12 (f(a_i^x)-f(a_{i'}^x))^2
}
\]
where $i' \equiv i+1 \mod 6$. We now define a function $f_{n+m}$ on $G^T_{n+m}$ such that $f_{n+m}|_{G_m^T} = \phi_m$ as follows: if $F_\omega: G^T_{n-1}\to G^T_{m+n}$ is the contraction mapping which takes $G_{n-1}^T$ to the $[a^x,a_i^x,a_{i'}^x]$, then
$f_{n+m}\circ F_\omega$ is equal to
\[
  (2\phi_m(a^x)-\phi_m(a_i^x)-\phi_m(a_{i'}^x))u+(\phi_m(a_{i}^x)-\phi_m(a_{i'}^x))(v-w) + \frac12(\phi_m(a_i^x)+\phi_m(a_{i'}^x))\mathbb{1}.
\]
and so
\begin{align*}
\eng_{n+m}^T(f_{n+m}) & = \sum_{\omega\in \set{0,\ldots,5}^{m+1}} \eng_{n-1}^T(f_{n+m}\circ F_\omega)\\
& = \sum_{x\in \E^{(2)}_n}\sum_{i=0}^5 (2\phi_m(a^x)-\phi_m(a_i^x)-\phi_m(a_{i'}^x))^2\eng^T_{n-1}(u) \\
& \phantom{=\sum \sum NN}+(\phi_m(a_{i}^x)-\phi_m(a_{i}^x))^2\eng^T_{n-1}(v-w)
\end{align*}
\begin{align*}
&\leq  \sum_{x\in \E^{(2)}_n}\sum_{i=0}^5 (2\phi_m(a^x)-\phi_m(a_i^x)-\phi_m(a_{i'}^x))^2\eng^T_{n-1}(u) \\
& \phantom{ = \sum \sum NN} +4(\phi_m(a_{i'}^x)-\phi_m(a_i^x))^2\eng^T_{n-1}(v)\\
&\leq  \sum_{x\in \E^{(2)}_n}\sum_{i=0}^5 2((\phi_m(a^x)-\phi_m(a_{i'}^x))^2+(\phi_m(a_i^x)-\phi_m(a_{i'}^x))^2)\eng^T_{n-1}(u)\\
 & \phantom{\leq \sum \sum NN}+4(\phi_m(a_{i'}^x)-\phi_m(a_i^x))^2\eng^T_{n-1}(v)\leq
\end{align*}
$$ 2\paren{2\eng^T_{n-1}(u){+}4\eng^T_{n-1}(w)}
\hskip-3pt  \sum_{x\in \E^{(2)}_n}\sum_{i=0}^5 \paren{(\phi_m(a^x){-}\phi_m(a_{i'}^x))^2{+}\frac12(\phi_m(a_i^x){-}\phi_m(a_{i'}^x))^2}.$$
This is less or equal to $2R_m^{-1}R_n^{-1}$, which implies that $R_{m+n} \geq R_mR_n/2$.



\subsection{Upper bound on $\rho$ by removing edges}\label{sec-short}

\begin{figure}
\inviz{
\hfill
\begin{tikzpicture}[scale=.8888]
\draw[dashed] (30:2.3)--(150:2.3);
\draw[thick,double] (150:2.3)--(270:2.3)--(30:2.3);
\draw[dashed] (-30:1.2)--(0,0)--(210:1.2);

\draw[double]  ($(0:.8)$)--($(60:.8)$);
\draw[double]  ($(60*1:.8)$)--($(60*1+60:.8)$);
\draw[double]  ($(60*2:.8)$)--($(60*2+60:.8)$);
\draw[double,dotted]  ($(60*3:.8)$)--($(60*4:.8)$);
\draw[double]  ($(60*4:.8)$)--($(60*4+60:.8)$);
\draw[double,dotted]  ($(60*5:.8)$)--($(0:.8)$);

\draw[dotted] ($(60*3+120*2:.8)$)--($(60*3+120*2:.8)+(30+60*3+120*2:.45)$);
\draw[dotted] ($(60+60*3+120*2:.8)$)--($(60+60*3+120*2:.8)+(30+60*3+120*2:.45)$);

\foreach \b in {0,1}{
\draw[double] ($(60*3+120*\b:.8)$)--($(60*3+120*\b:.8)+(30+60*3+120*\b:.45)$);
\draw[double] ($(60+60*3+120*\b:.8)$)--($(60+60*3+120*\b:.8)+(30+60*3+120*\b:.45)$);
}
\foreach \b in {0,1,2}{
\draw[fill=white] ($(60*3+120*\b:.8)+(30+60*3+120*\b:.45)$)      circle (3pt);
\draw[fill=white] ($(60+60*3+120*\b:.8)+(30+60*3+120*\b:.45)$)   circle (3pt);
}

\foreach \b in {0,1,2,...,6}{
\draw[fill=white] ($(60*\b:.8)$)                    circle (3pt);
}

\node at (180:1.8) {$L_1^{(1)}$};
\node at (240:1.8) {$L_0^{(1)}$};
\node at (300:1.8) {$L_5^{(1)}$};
\node at (0:1.8) {$L_4^{(1)}$};
\node at (90:.4) {$l_{1,2}$};
\node at (270:1) {$l_{1,1}$};

\end{tikzpicture}
\hfill
\begin{tikzpicture}[scale=.8888]

\node at (150:3.8) {$L_1^{(2)}$};
\node at (210:3.9) {$L_0^{(2)}$};
\node at (270:3.8) {$L_5^{(2)}$};
\node at (330:3.8) {$L_4^{(2)}$};
\node at (70:2.9) {$l_{2,4}$};
\node at (75:1.1) {$l_{2,3}$};
\node at (225:1.1) {$l_{2,2}$};
\node at (230:2.9) {$l_{2,1}$};

\foreach \a in {0,1}{
\draw[dashed] ($(60*\a:4)$)--($(60+60*\a:4)$);
\draw[dashed] (0,0)--($(-180+120*\a:4)$);

}
\foreach \a in {1,3,4,5}{
\draw[double] (0,0)--($(-180+60*\a:4)$);
}

\draw[thick,double] (-120:4)--(-60:4)--(0:4);
\draw[thick,double] (120:4)--(180:4)--(240:4);

\foreach \a in {0,1,3,5}{
\draw[dashed] ($(30+60*\a:2.31)$)--($(90+60*\a:2.31)$);
}

\foreach \a in {2,3,4,5}{
\draw[dashed] ($(30+60*\a:2.31)$)--($(30+60*\a:3.4)$);
}

\foreach \a in {0,1,...,6}{
\foreach \b in {0,1,2,...,6}{

\draw[dotted] ($(30+60*\a:2.31)+(60*\b:.8)$)--($(30+60*\a:2.31)+(60+60*\b:.8)$);
\draw[dotted] ($(30+60*\a:2.31)+(180+60*\a:.8)$)--($(90+60*\a:2.31)+(-60+60*\a:.8)$);
\draw[dotted] ($(30+60*\a:2.31)+(120+60*\a:.8)$)--($(90+60*\a:2.31)+(60*\a:.8)$);
\draw[dotted] ($(30+60*\a:2.31)+(60*\a:.8)$)--($(30+60*\a:2.31)+(60*\a:.8)+(30+60*\a:.45)$);
\draw[dotted] ($(30+60*\a:2.31)+(60+60*\a:.8)$)--($(30+60*\a:2.31)+(60+60*\a:.8)+(30+60*\a:.45)$);
}}

\foreach \a in {0,1,3,5}{
\foreach \b in {0,1,2,...,6}{

\draw ($(30+60*\a:2.31)+(180+60*\a:.8)$)--($(90+60*\a:2.31)+(-60+60*\a:.8)$);
\draw ($(30+60*\a:2.31)+(120+60*\a:.8)$)--($(90+60*\a:2.31)+(60*\a:.8)$);
}}

\foreach \a in {2,3,4,5}{
\foreach \b in {0,1,2,...,6}{

\draw ($(30+60*\a:2.31)+(60*\a:.8)$)--($(30+60*\a:2.31)+(60*\a:.8)+(30+60*\a:.45)$);
\draw ($(30+60*\a:2.31)+(60+60*\a:.8)$)--($(30+60*\a:2.31)+(60+60*\a:.8)+(30+60*\a:.45)$);
}}

\draw  ($(30+60*0:2.31)+(60*0:.8)$)--($(30+60*0:2.31)+(60+60*0:.8)$);\draw  ($(30:2.31)+(-60:.8)$)--($(30:2.31)+(60-60:.8)$);
\draw  ($(30:2.31)+(60*3:.8)$)--($(30:2.31)+(60+60*3:.8)$);
\draw  ($(30:2.31)+(60:.8)$)--($(30:2.31)+(60+60:.8)$);
\draw ($(30+60:2.31)+(60:.8)$)--($(30+60:2.31)+(60+60:.8)$);

\draw  ($(30+60*1:2.31)+(0:.8)$)--($(30+60*1:2.31)+(60:.8)$);
\draw  ($(30+60*1:2.31)+(60*1:.8)$)--($(30+60*1:2.31)+(60*1+60:.8)$);
\draw  ($(30+60*1:2.31)+(60*4:.8)$)--($(30+60*1:2.31)+(60*4+60:.8)$);
\draw  ($(30+60*1:2.31)+(60*2:.8)$)--($(30+60*1:2.31)+(60*2+60:.8)$);

\draw  ($(30+60*2:2.31)+(-60*0:.8)$)--($(30+60*2:2.31)+(-60*0-60:.8)$);
\draw  ($(30+60*2:2.31)+(-60*1:.8)$)--($(30+60*2:2.31)+(-60*1-60:.8)$);
\draw  ($(30+60*2:2.31)+(-60*4:.8)$)--($(30+60*2:2.31)+(-60*4-60:.8)$);
\draw  ($(30+60*2:2.31)+(-60*2:.8)$)--($(30+60*2:2.31)+(-60*2-60:.8)$);

\draw  ($(30+60*3:2.31)+(0:.8)$)--($(30+60*3:2.31)+(60:.8)$);
\draw  ($(30+60*3:2.31)+(60*1:.8)$)--($(30+60*3:2.31)+(60*1+60:.8)$);
\draw  ($(30+60*3:2.31)+(60*4:.8)$)--($(30+60*3:2.31)+(60*4+60:.8)$);
\draw  ($(30+60*3:2.31)+(60*2:.8)$)--($(30+60*3:2.31)+(60*2+60:.8)$);

\draw  ($(30+60*4:2.31)+(60*0:.8)$)--($(30+60*4:2.31)+(60+60*0:.8)$);
\draw  ($(30+60*4:2.31)+(-60:.8)$)--($(30+60*4:2.31)+(60-60:.8)$);
\draw  ($(30+60*4:2.31)+(60*3:.8)$)--($(30+60*4:2.31)+(60+60*3:.8)$);
\draw  ($(30+60*4:2.31)+(60:.8)$)--($(30+60*4:2.31)+(60+60:.8)$);

\draw  ($(30+60*5:2.31)+(60*0:-.8)$)--($(30+60*5:2.31)+(60+60*0:-.8)$);
\draw  ($(30+60*5:2.31)+(-60:-.8)$)--($(30+60*5:2.31)+(60-60:-.8)$);
\draw  ($(30+60*5:2.31)+(60*3:-.8)$)--($(30+60*5:2.31)+(60+60*3:-.8)$);
\draw  ($(30+60*5:2.31)+(60:-.8)$)--($(30+60*5:2.31)+(60+60:-.8)$);

\draw  ($(30+60*2:2.31)+(60:.8)$)--($(30+60*2:2.31)+(60+60:.8)$);


\foreach \a in {0,1,...,6}{
\foreach \b in {0,1,2,...,6}{
\draw[fill=white] ($(30+60*\a:2.31)+(60*\b:.8)$)                     circle (2.2222pt);
\draw[fill=white] ($(30+60*\a:2.31)+(60*\a:.8)+(30+60*\a:.45)$)      circle (2.2222pt);
\draw[fill=white] ($(30+60*\a:2.31)+(60+60*\a:.8)+(30+60*\a:.45)$)   circle (2.2222pt);
}}
\hfill~
\end{tikzpicture}
}
\caption{Gluing from $\widehat{G_1}$ to $\widehat{G_2}$.}\label{CutGraph}
\end{figure}
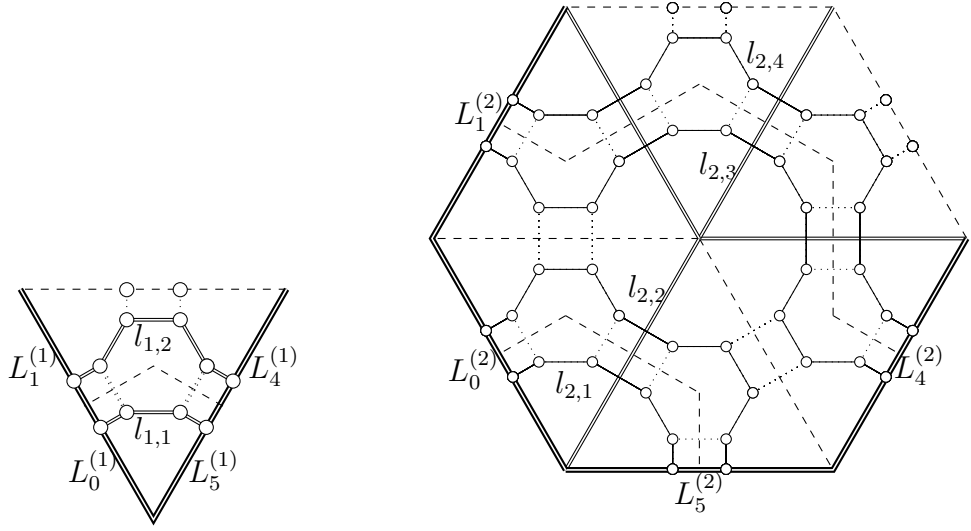


\begin{prop}
$R_n\leq (3/2)^n$, and thus $\rho\leq 3/2$.
\end{prop}

\begin{proof}
Figure~\ref{CutGraph} shows $\widehat G_1 = (V^H_1,\widehat E_1)$, which is obtained from  $G^H_1$ by removing all edges which have the vertexes contained in  $|[p'_0,p']|$ and $|[p_2',p']|$.
Figure~\ref{CutGraph} also illustrates how six $\widehat{G}_n$ graphs can be glued together to form one $\widehat{G}_{n+1}$ graph.  By induction, we see that each $\widehat{G}_n$ is made up of $2^n$ paths between $L_0^{(n,H)}\cup L_1^{(n,H)}$ and $L_5^{(n,H)}\cup L_4^{(n,H)}$.



The lengths of all the paths from $L_0^{(n,H)}\cup L_1^{(n,H)}$ and $L_5^{(n,H)}\cup L_4^{(n,H)}$ in $\widehat{G}_n$ can be encoded in a sequence of $2^n$ integers. We will call this sequence $l_n$ and write $l_n=\{l_{n,1},l_{n,2},...,l_{n,2^n}\}$ where $l_{n,j}$ is the length of the path which has initial (or terminal) point $j^{th}$ closest to $p_0$ (in graph or Euclidean distance with our embedding). Since $\widehat G_{n+1}$ is 6 copies of $\widehat G_{n}$ glued together, $\sum_{k=1}^{2^n} l_{n,k}  = 6 \sum_{k=1}^{2^{n-1}}l_{n-1,k} = 6^n$, because $l_1 = \{2,4\}$.

%

The corresponding resistance $\widehat{R}_n$ between $L_0^{(n,H)}\cup L_1^{(n,H)}$ and $L_5^{(n,H)}\cup L_4^{(n,H)}$ of the $\widehat{G}_n$ graph is the resistance of $2^n$ paths connected in parallel so, by Kirchhoff's laws,
$
\widehat{R}_n=\left(\sum_{j=1}^{2^n}\frac{1}{l_{n,j}}\right)^{-1}
$.
Using Jensen's inequality for the
convex function $x\mapsto 1/x$    on $(0,\infty)$, we have
\[
\widehat R_n=\frac{1}{2^n}\left(\frac{1}{2^n}\sum_{j=1}^{2^n}\frac{1}{l_{n,j}}\right)^{-1}\leq \frac{1}{2^n}
\frac{1}{2^n}\sum_{j=1}^{2^n}l_{n,j}
=\left(\frac{3}{2}\right)^n.
\]
To obtain an upper bound on $R_n$, we glue together 6 copies of $\widehat G_{n-1}$ to produce a graph which has an edge set contained in $E^H_n$  consisting of $2^n$ paths from $A_n^H$ to $B_n^H$. Then, using
Kirchhoff's laws
and
the above argument, we obtain that we are connecting $A_n^H$ to $B_n^H$ with $2$ parallel connections of three sequential wires of resistance $\widehat{R}_n$. Thus we have that $R_n \leq 3R_{n-1}/2 = (3/2)^n$.
\end{proof}

This also implies a lower bound on ${\rho^T} \geq 2/3$, which can also be obtained by considering the graph $\widetilde G^T_n$ with vertex set $\widetilde V_n^T$ where $x,y\in V_n^T$ (the vertex set of $G_n^T$) are identified as one vertex if an edge connecting $x$ and $y$ was deleted in the construction of $\widehat G_n$ with ``end'' points $P$ and $Q$, which correspond to the points in $A^T$ and $B^T$,
Figure~\ref{tildeGT}:

 \begin{figure}[hbt]

 \vskip-.5ex

{
\hfil~
\begin{tikzpicture}[scale=.88]
\draw[fill=black] (0,0) node(a) {} circle (4pt);
\draw (0,-.4) node {$P$};
\draw[fill=black] (2,0) node(b) {} circle (4pt);
\draw[fill=black] (4,0) node(c) {} circle (4pt);
\draw (4,-.4) node {$Q$};

\draw	  (a) to  [out=40,in=140] (b);
\draw	  (a) to [out=-40,in=-140] node[below] {$l_{1,1}$}(b);

\draw	  (b) to  [out=60,in=120] (c);
\draw	  (b) to [out=-60,in=-120]node[below] {$l_{1,2}$} (c);
\draw	  (b) to  [out=30,in=150] (c);
\draw	  (b) to [out=-30,in=-150] (c);
\end{tikzpicture}
\hfil~
\begin{tikzpicture}[scale=.88]
\draw (0,0)--(2,0)--(4,0)--(6,0)--(8,0);
\draw[fill=black] (0,0) node(a) {} circle (4pt);
\draw (0,-.4) node {$P$};
\draw[fill=black] (2,0) node(b) {} circle (4pt);
\draw[fill=black] (4,0) node(c) {} circle (4pt);
\draw[fill=black] (6,0) node(d) {} circle (4pt);
\draw[fill=black] (8,0) node(e) {} circle (4pt);
\draw (8,-.4) node {$Q$};

\draw	  (a) to  [out=40,in=140] (b);
\draw	  (a) to [out=-40,in=-140]node[below] {$l_{2,1}$} (b);
\draw	  (a) to  [out=20,in=160] (b);
\draw	  (a) to [out=-20,in=-160] (b);

\draw  	  (b) to  [out=40,in=140] (c)
		  (b) to  [out=-40,in=-140] node[below] {$l_{2,2}$} (c)
    	  (b) to  [out=10,in=170] (c)
		  (b) to  [out=-10,in=-170] (c)
    	  (b) to  [out=20,in=160] (c)
		  (b) to  [out=-20,in=-160](c)
    	  (b) to  [out=30,in=150] (c)
		  (b) to  [out=-30,in=-150] (c);
		
\draw  	  (c) to  [out=40,in=140] (d)
		  (c) to  [out=-40,in=-140] (d)
    	  (c) to  [out=10,in=170] (d)
		  (c) to  [out=-10,in=-170] (d)
    	  (c) to  [out=20,in=160] (d)
		  (c) to  [out=-20,in=-160] (d)
      	  (c) to  [out=50,in=130] (d)
		  (c) to  [out=-50,in=-130] (d)
      	  (c) to  [out=60,in=120] (d)
		  (c) to  [out=-60,in=-120] node[below] {$l_{2,3}$} (d)
    	  (c) to  [out=30,in=150] (d)
		  (c) to  [out=-30,in=-150] (d);
		
\draw  	  (d) to  [out=40,in=140] (e)
		  (d) to  [out=-40,in=-140] (e)
    	  (d) to  [out=10,in=170] (e)
		  (d) to  [out=-10,in=-170] (e)
    	  (d) to  [out=20,in=160] (e)
		  (d) to  [out=-20,in=-160] (e)
      	  (d) to  [out=50,in=130] (e)
		  (d) to  [out=-50,in=-130] (e)
      	  (d) to  [out=60,in=120] (e)
		  (d) to  [out=-60,in=-120] node[below] {$l_{2,4}$}(e)
    	  (d) to  [out=30,in=150] (e)
		  (d) to  [out=-30,in=-150] (e);
\end{tikzpicture}
\hfil~}

 \vskip-1.5ex

\caption{
 Short-circuited graphs     $\widetilde G^T_1$ and $\widetilde G^T_2$.
 \label{tildeGT}}
 \end{figure}
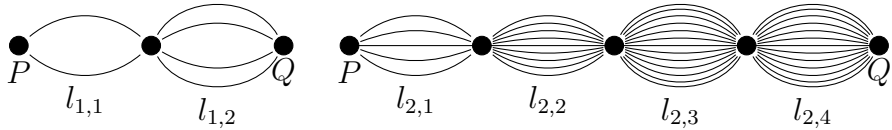

\subsection{Lower bound on $\rho$ by shorting graph}\label{sec-short2}

In this subsection we define a graph $\widetilde G_n^H$, such that $\widetilde V_n$ is $V_n^H$ modulo an equivalence relation. Thus, resistance in $\widetilde G_n^H$ between two sets is less than the resistance between the fibers of these sets.

\begin{prop}
  $R_n \geq c(5/4)^n$ for a constant $c$ independent of $n$, and so $\rho\geq 5/4$.
\end{prop}

\begin{proof}
We define that two vertexes in $\widetilde V^H_n$ to be equivalent if they are both contained in $F_\omega([p_i,p_{j}])$ where $j\equiv i+1 \mod 3$ for some $i$ and some $\omega \in \set{0,1,\ldots, 5}^k$. Figure~\ref{ShortGraph} shows $\widetilde G_2^H$.
%
These graphs appeared in \cite{Tep08}, as an example of non-p.c.f. Sierpinski gaskets, where it was determined that their resistance scaling factor is $\frac{5}{4}$. This implies that $\widetilde{R}_{n+1}=\frac{5}{4}\widetilde{R}_n$ and subsequently $\widetilde{R}_n=\widetilde{R}_1(\frac{5}{4})^{n-1}$. From this, and a gluing argument as in the previous subsection, it follows that $R_n\geq c (\frac{5}{4})^n$.
\end{proof}


Using a similar argument, $G^T_n$ can be obtained by identifying points in graph approximations of another non-p.c.f. Sierpinski gasket which appears at the end of \cite{Tep08}, and which is pictured in Figure~\ref{cutgraph}. The resistance between the corner points of these graphs is some constant times $(4/5)^n$, and thus the resistance between the corner points of $G^T_n$ is less than this value. This proves that $\rho^T \leq 4/5$ because including more points in the boundary decreases resistance, so the resistance between the corner points is greater than the resistance between $A$ and $B$.
Alternatively, if we connect the boundary points of the graph in the left of Figure~\ref{cutgraph} to a $\triangle$-network, it is dual to the network attained by connecting the corner points in the graph in the right of Figure~\ref{cutgraph} to a $Y$-network. This aslo explains why the resistances are reciprocal.

\begin{figure}
\inviz{
\begin{tikzpicture}[scale=.7777]


\foreach \a in {0,1,...,6}{
\foreach \b in {0,1,2,...,6}{

\draw ($(30+60*\a:2.31)+(60*\b:.8)$)--($(30+60*\a:2.31)+(60+60*\b:.8)$);
\draw ($(30+60*\a:2.31)+(180+60*\a:.8)$)--($(90+60*\a:2.31)+(60*\a:.8)$);
\draw ($(30+60*\a:2.31)+(120+60*\a:.8)$)--($(90+60*\a:2.31)+(-60+60*\a:.8)$);
\draw ($(30+60*\a:2.31)$)node {$\widetilde G_1^H$};
}}

\foreach \a in {0,1,2}{

\draw ($(30+120*\a:2.31)+(120*\a:.8)$)to[out=55+120*\a,in=-55+120*\a] ($(60+120*\a:4)$);
\draw ($(30+120*\a:2.31)+(60+120*\a:.8)$)to[out=65+120*\a,in=-65+120*\a]($(60+120*\a:4)$) ;
\draw ($(30+120*\a+60:2.31)+(120*\a+60:.8)$)to[out=55+120*\a,in=200+120*\a]($(60+120*\a:4)$);
\draw ($(30+120*\a+60:2.31)+(60+120*\a+60:.8)$)to[out=65+120*\a,in=185+120*\a]($(60+120*\a:4)$) ;
\draw[fill=white]($(60+120*\a:4)$) circle (2.2222pt);
}

\foreach \a in {0,1,...,6}{
\foreach \b in {0,1,2,...,6}{
\draw[fill=white] ($(30+60*\a:2.31)+(60*\b:.8)$)                     circle (2.2222pt);
\draw[fill=white] ($(-120+60*\a:2)$) circle (2.2222pt);
}}

\end{tikzpicture}
\quad
\begin{tikzpicture}[scale=.7777]

\foreach \a in {0,1,2}{
\draw   ($(90+120*\a:4)$)to[out=280+120*\a,in=80+120*\a] (0,0);
\draw   ($(90+120*\a:4)$)to[out=260+120*\a,in=100+120*\a] (0,0);
\draw   (0,0)to[out=-105+120*\a,in=105+120*\a] ($(-90+120*\a:2)$);
\draw   (0,0)to[out=-75+120*\a,in=75+120*\a] ($(-90+120*\a:2)$);
\draw (0:0)--($.5*(90+120*\a:4)+.5*(30+120*\a:2)$);
\draw (0:0)--($.5*(-30+120*\a:4)+.5*(30+120*\a:2)$);
\draw ($(90+120*\a:4)$)--($(38+120*\a:1)$);
\draw ($(90+120*\a:4)$)--($(142+120*\a:1)$);
\draw ($(-36+120*\a:1.85)$)--($(-90+120*\a:2)$);
\draw ($(216+120*\a:1.85)$)--($(-90+120*\a:2)$);
}

\draw (90:4)--(210:4)--(330:4) -- cycle;

\end{tikzpicture}
}%
\caption{Left: Graph $\widetilde G_2^H$ with short circuits. Right: Non-p.c.f Sierpinski gasket.}
\label{cutgraph}\label{ShortGraph}\vspace*{-10pt}
\end{figure}
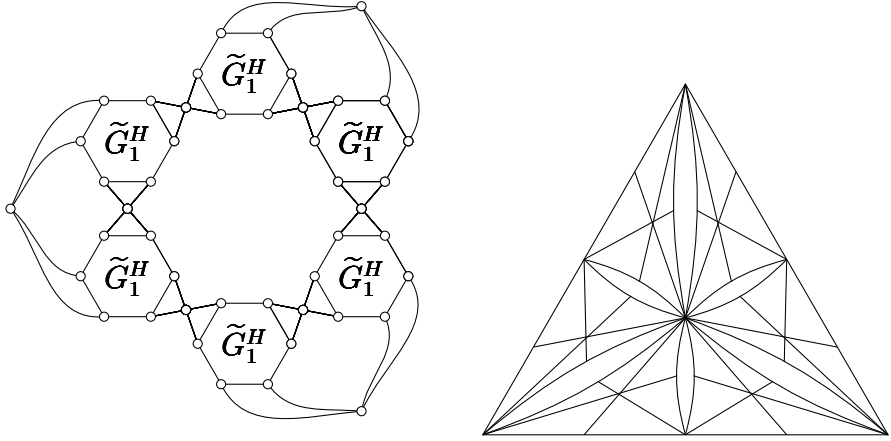

\section*{Acknowledgment}The authors  thank the anonymous reviewers for their valuable comments and suggestions to improve the quality of the paper.
%

\def\cprime{$'$}
\providecommand{\href}[2]{#2}
\providecommand{\arxiv}[1]{\href{http://arxiv.org/abs/#1}{arXiv:#1}}
\providecommand{\url}[1]{\texttt{#1}}
\providecommand{\urlprefix}{URL }


\begin{thebibliography}{10}
\def\doititle#1{\emph{#1}}
\def\arXiv#1{\href{http://arxiv.org/abs/#1}{arXiv:#1}}

\bibitem{Am} 
\newblock L.~Ambrosio, M.~Erbar and G.~Savar{\'e},
\newblock \doititle{Optimal transport, {C}heeger energies and contractivity of dynamic transport distances in extended spaces},
\newblock \emph{Nonlinear Anal.}, \textbf{137} (2016), 77--134.

\bibitem{Barlow} 
\newblock M.~T. Barlow,
\newblock Analysis on the {S}ierpinski carpet,
\newblock in \emph{Analysis and geometry of metric measure spaces}, vol.~56 of CRM Proc. Lecture Notes, Amer. Math. Soc., Providence, RI, 2013, 27--53.

\bibitem{BB90} 
\newblock M.~T. Barlow and R.~F. Bass,
\newblock \doititle{On the resistance of the {S}ierpi\'nski carpet},
\newblock \emph{Proc. Roy. Soc. London Ser. A}, \textbf{431} (1990), 345--360.

\bibitem{BBS90} 
\newblock M.~T. Barlow, R.~F. Bass and J.~D. Sherwood,
\newblock \doititle{Resistance and spectral dimension of {S}ierpi\'nski carpets},
\newblock \emph{J. Phys. A}, \textbf{23} (1990), L253--L258.

\bibitem{Barlow98} 
\newblock M.~Barlow,
\newblock \doititle{Diffusions on fractals},
\newblock in \emph{Lectures on Probability Theory and Statistics ({S}aint-{F}lour, 1995)}, vol. 1690 of Lecture Notes in Math., Springer, Berlin, 1998, 1--121.

\bibitem{BB99} 
\newblock M.~Barlow and R.~Bass,
\newblock \doititle{Brownian motion and harmonic analysis on {S}ierpinski carpets},
\newblock \emph{Canad. J. Math.}, \textbf{51} (1999), 673--744.

\bibitem{BBKT} 
\newblock M.~Barlow, R.~F. Bass, T.~Kumagai and A.~Teplyaev,
\newblock Uniqueness of {B}rownian motion on {S}ierpi\'nski carpets,
\newblock \emph{J. Eur. Math. Soc. (JEMS)}, \textbf{12} (2010), 655--701.

\bibitem{BGN} 
\newblock L.~Bartholdi, R.~Grigorchuk and V.~Nekrashevych,
\newblock From fractal groups to fractal sets,
\newblock in \emph{Fractals in {G}raz 2001}, Trends Math., Birkh\"auser, Basel, 2003, 25--118.

\bibitem{Bass} 
\newblock R.~Bass,
\newblock Diffusions on the {S}ierpinski carpet,
\newblock in \emph{Trends in Probability and Related Analysis ({T}aipei, 1996)}, World Sci. Publ., River Edge, NJ, 1997, 1--34.

\bibitem{BK16} [10.1090/tran/7362]
\newblock F.~Baudoin and D.~J. Kelleher,
\newblock \doititle{Differential one-forms on dirichlet spaces and bakry-emery estimates on metric graphs},
\newblock \arXiv{1604.02520}, \emph{Transactions of the AMS}, to appear.

\bibitem{KW1} 
\newblock F.~Bauer, M.~Keller and R.~K. Wojciechowski,
\newblock \doititle{Cheeger inequalities for unbounded graph {L}aplacians},
\newblock \emph{J. Eur. Math. Soc. (JEMS)}, \textbf{17} (2015), 259--271.

\bibitem{BKN+12} 
\newblock M.~Begue, D.~Kelleher, A.~Nelson, H.~Panzo, R.~Pellico and A.~Teplyaev,
\newblock \doititle{Random walks on barycentric subdivisions and the {S}trichartz hexacarpet},
\newblock \emph{Exp. Math.}, \textbf{21} (2012), 402--417.

\bibitem{Str3} 
\newblock R.~Bell, C.-W. Ho and R.~S. Strichartz,
\newblock \doititle{Energy measures of harmonic functions on the {S}ierpi\'nski gasket},
\newblock \emph{Indiana Univ. Math. J.}, \textbf{63} (2014), 831--868.

\bibitem{Str2} 
\newblock J.~Bello, Y.~Li and R.~S. Strichartz,
\newblock Hodge--de {R}ham theory of {K}-forms on carpet type fractals,
\newblock in \emph{Excursions in Harmonic Analysis}, Appl. Numer. Harmon. Anal., Birkh\"auser/Springer, Cham, \textbf{3} (2015), 23--62.

\bibitem{BH} 
\newblock N.~Bouleau and F.~Hirsch,
\newblock \emph{Dirichlet Forms and Analysis on {W}iener Space}, vol.~14 of de Gruyter Studies in Mathematics,
\newblock Walter de Gruyter \& Co., Berlin, 1991.

\bibitem{ChenF} 
\newblock Z.-Q. Chen and M.~Fukushima,
\newblock \emph{Symmetric {M}arkov Processes, Time Change, and Boundary Theory}, vol.~35 of London Mathematical Society Monographs Series,
\newblock Princeton Univ. Press, 2012.

\bibitem{DF99} 
\newblock P.~Diaconis and D.~Freedman,
\newblock \doititle{Iterated random functions},
\newblock \emph{SIAM Rev.}, \textbf{41} (1999), 45--76.

\bibitem{DMcM}
\newblock P.~Diaconis and C.~McMullen,
\newblock \doititle{Barycentric Subdivision},
\newblock Unpublished, 2008.

\bibitem{DMic}
\newblock P.~Diaconis and L.~Miclo,
\newblock On barycentric partitions, with simulations,
\newblock \url{https://hal.archives-ouvertes.fr/hal-00353842}.

\bibitem{DM11} 
\newblock P.~Diaconis and L.~Miclo,
\newblock \doititle{On barycentric subdivision},
\newblock \emph{Combin. Probab. Comput.}, \textbf{20} (2011), 213--237.

\bibitem{DS84} 
\newblock P.~G. Doyle and J.~L. Snell,
\newblock \emph{Random Walks and Electric Networks}, vol.~22 of Carus Mathematical Monographs,
\newblock Mathematical Association of America, Washington, DC, 1984.

\bibitem{FOT11} 
\newblock M.~Fukushima, Y.~Oshima and M.~Takeda,
\newblock \emph{Dirichlet Forms and Symmetric {M}arkov Processes}, vol.~19 of de Gruyter Studies in Mathematics,
\newblock extended edition, Walter de Gruyter \& Co., Berlin, 2011.

\bibitem{GN} 
\newblock R.~Grigorchuk and V.~Nekrashevych,
\newblock \doititle{Self-similar groups, operator algebras and {S}chur complement},
\newblock \emph{J. Mod. Dyn.}, \textbf{1} (2007), 323--370.

\bibitem{GH} 
\newblock A.~Grigor'yan and J.~Hu,
\newblock \doititle{Heat kernels and {G}reen functions on metric measure spaces},
\newblock \emph{Canad. J. Math.}, \textbf{66} (2014), 641--699.

\bibitem{GHL1}
Grigor'yan, A., Hu Jiaxin., Lau, K.-S., \emph{Generalized capacity, Harnack inequality and heat kernels of Dirichlet forms on metric spaces}, J. Math. Soc. Japan, 67 (2015) 1485--1549.

\bibitem{GHL2}
Grigor'yan, A., Hu Jiaxin, Lau, K.-S. \emph{Estimates of heat kernels for non-local regular Dirichlet forms}, Trans. AMS,  366 (2014), 6397--6441.



\bibitem{GT} 
\newblock A.~Grigor'yan and A.~Telcs,
\newblock \doititle{Two-sided estimates of heat kernels on metric measure spaces},
\newblock \emph{Ann. Probab.}, \textbf{40} (2012), 1212--1284.

\bibitem{GY2018}  
\newblock A.~Grigor'yan and M.~Yang,
\newblock \doititle{Local and non-local Dirichlet forms on the Sierpinski carpet},
\newblock preprint.

\bibitem{HMT06} 
\newblock B.~M. Hambly, V.~Metz and A.~Teplyaev,
\newblock \doititle{Self-similar energies on post-critically finite self-similar fractals},
\newblock \emph{J. London Math. Soc. (2)}, \textbf{74} (2006), 93--112.

\bibitem{HSTZjulia} 
\newblock K.~E. Hare, B.~A. Steinhurst, A.~Teplyaev and D.~Zhou,
\newblock \doititle{Disconnected {J}ulia sets and gaps in the spectrum of {L}aplacians on symmetric finitely ramified fractals},
\newblock \emph{Math. Res. Lett.}, \textbf{19} (2012), 537--553.

\bibitem{Hatcher} 
\newblock A.~Hatcher,
\newblock \emph{Algebraic Topology},
\newblock Cambridge University Press, Cambridge, 2002.

\bibitem{book} 
\newblock J.~Heinonen, P.~Koskela, N.~Shanmugalingam and J.~T. Tyson,
\newblock \emph{Sobolev Spaces on Metric Measure Spaces}, vol.~27 of New Mathematical Monographs,
\newblock Cambridge University Press, Cambridge, 2015, An approach based on upper gradients.

\bibitem{HKT12} 
\newblock M.~Hinz, D.~Kelleher and A.~Teplyaev,
\newblock Measures and {D}irichlet forms under the {G}elfand transform,
\newblock \emph{Zap. Nauchn. Sem. S.-Peterburg. Otdel. Mat. Inst. Steklov. (POMI)}, \textbf{408} (2012), 303--322, 329--330.

\bibitem{HT15} 
\newblock M.~Hinz and A.~Teplyaev,
\newblock Closability, regularity, and approximation by graphs for separable bilinear forms,
\newblock \emph{Zap. Nauchn. Sem. S.-Peterburg. Otdel. Mat. Inst. Steklov. (POMI)}, \textbf{441} (2015), 299--317.

\bibitem{HKT13} 
\newblock M.~Hinz, D.~J. Kelleher and A.~Teplyaev,
\newblock \doititle{Metrics and spectral triples for {D}irichlet and resistance forms},
\newblock \emph{J. Noncommut. Geom.}, \textbf{9} (2015), 359--390.

\bibitem{HLTV16} 
\newblock M.~Hinz, M.~R. Lacia, A.~Teplyaev and P.~Vernole,
\newblock \doititle{Fractal snowflake domain diffusion with boundary and interior drifts},
\newblock \emph{J. Math. Anal. Appl.,} \textbf{457} (2018), 672--693, \arXiv{1605.06785}.

\bibitem{HRT} 
\newblock M.~Hinz, M.~R{\"o}ckner and A.~Teplyaev,
\newblock \doititle{Vector analysis for {D}irichlet forms and quasilinear {PDE} and {SPDE} on metric measure spaces},
\newblock \emph{Stochastic Process. Appl.}, \textbf{123} (2013), 4373--4406.

\bibitem{HT12} 
\newblock M.~Hinz and A.~Teplyaev,
\newblock \doititle{Dirac and magnetic {S}chr\"odinger operators on fractals},
\newblock \emph{J. Funct. Anal.}, \textbf{265} (2013), 2830--2854.

\bibitem{IRT12} 
\newblock M.~Ionescu, L.~Rogers and A.~Teplyaev,
\newblock \doititle{Derivations and {D}irichlet forms on fractals},
\newblock \emph{J. Funct. Anal.}, \textbf{263} (2012), 2141--2169.

\bibitem{Kai05} 
\newblock V.~A. Kaimanovich,
\newblock \doititle{``{M}\"unchhausen trick'' and amenability of self-similar groups},
\newblock \emph{Internat. J. Algebra Comput.}, \textbf{15} (2005), 907--937.

\bibitem{Kajino} 
\newblock N.~Kajino,
\newblock \doititle{Heat kernel asymptotics for the measurable {R}iemannian structure on the {S}ierpinski gasket},
\newblock \emph{Potential Anal.}, \textbf{36} (2012), 67--115.

\bibitem{Kajino13} 
\newblock N.~Kajino,
\newblock \doititle{Analysis and geometry of the measurable {R}iemannian structure on the {S}ierpi\'nski gasket},
\newblock in \emph{Fractal Geometry and Dynamical Systems in Pure and Applied Mathematics. {I}. {F}ractals in Pure Mathematics}, vol. 600 of Contemp. Math., Amer. Math. Soc., Providence, RI, 2013, 91--133.

\bibitem{SteinhurstQM} 
\newblock C.~J. Kauffman, R.~M. Kesler, A.~G. Parshall, E.~A. Stamey and B.~A. Steinhurst,
\newblock \doititle{Quantum mechanics on {L}aakso spaces},
\newblock \emph{J. Math. Phys.}, \textbf{53} (2012), 042102, 18pp.

\bibitem{KSWpcf} 
\newblock D.~J. Kelleher, B.~A. Steinhurst and C.-M.~M. Wong,
\newblock \doititle{From self-similar structures to self-similar groups},
\newblock \emph{Internat. J. Algebra Comput.}, \textbf{22} (2012), 1250056, 16pp.

\bibitem{KW2} 
\newblock M.~Keller, D.~Lenz and R.~K. Wojciechowski,
\newblock \doititle{Volume growth, spectrum and stochastic completeness of infinite graphs},
\newblock \emph{Math. Z.}, \textbf{274} (2013), 905--932.

\bibitem{Kig93} 
\newblock J.~Kigami,
\newblock \doititle{Harmonic calculus on p.c.f.\ self-similar sets},
\newblock \emph{Trans. Amer. Math. Soc.}, \textbf{335} (1993), 721--755.

\bibitem{Kig01} 
\newblock J.~Kigami,
\newblock \emph{Analysis on Fractals}, vol. 143 of Cambridge Tracts in Mathematics,
\newblock Cambridge University Press, Cambridge, 2001.

\bibitem{Ki3} 
\newblock J.~Kigami,
\newblock \doititle{Local {N}ash inequality and inhomogeneity of heat kernels},
\newblock \emph{Proc. London Math. Soc. (3)}, \textbf{89} (2004), 525--544.

\bibitem{Ki2} 
\newblock J.~Kigami,
\newblock \doititle{Volume doubling measures and heat kernel estimates on self-similar sets},
\newblock \emph{Mem. Amer. Math. Soc.}, \textbf{199} (2009), viii+94pp.

\bibitem{Ki1} 
\newblock J.~Kigami,
\newblock \doititle{Quasisymmetric modification of metrics on self-similar sets},
\newblock in \emph{Geometry and Analysis of Fractals}, vol.~88 of Springer Proc. Math. Stat., Springer, Heidelberg, 2014, 253--282.

\bibitem{Knill1}
\newblock O.~Knill,
\newblock The graph spectrum of barycentric refinements,
\newblock \arXiv{1508.02027}.

\bibitem{Knill2}
\newblock O.~Knill,
\newblock Universality for Barycentric subdivision,
\newblock \arXiv{1509.06092}.

\bibitem{KZ92} 
\newblock S.~Kusuoka and X.~Y. Zhou,
\newblock \doititle{Dirichlet forms on fractals: {P}oincar\'e constant and resistance},
\newblock \emph{Probab. Theory Related Fields}, \textbf{93} (1992), 169--196.

\bibitem{LS14} 
\newblock M.~Lapidus and J.~Sarhad,
\newblock \doititle{Dirac operators and geodesic metric on the harmonic {S}ierpinski gasket and other fractal sets},
\newblock \emph{J. Noncommut. Geom.}, \textbf{8} (2014), 947--985.

\bibitem{Li86} 
\newblock P.~Li,
\newblock \doititle{Large time behavior of the heat equation on complete manifolds with nonnegative {R}icci curvature},
\newblock \emph{Ann. of Math. (2)}, \textbf{124} (1986), 1--21.

\bibitem{LY86} 
\newblock P.~Li and S.-T. Yau,
\newblock \doititle{On the parabolic kernel of the {S}chr\"odinger operator},
\newblock \emph{Acta Math.}, \textbf{156} (1986), 153--201.

\bibitem{Lind} 
\newblock T.~Lindstr{\o}m,
\newblock \doititle{Brownian motion on nested fractals},
\newblock \emph{Mem. Amer. Math. Soc.}, \textbf{83} (1990), iv+128pp.

\bibitem{SteinhurstBond} 
\newblock D.~Lougee and B.~Steinhurst,
\newblock \doititle{Bond percolation on a non-{P}.{C}.{F}. {S}ierpi\'nski gasket, iterated barycentric subdivision of a triangle, and hexacarpet},
\newblock \emph{Fractals}, \textbf{24} (2016), 1650023, 12pp.

\bibitem{Lyo14} 
\newblock R.~Lyons and Y.~Peres,
\newblock \emph{Probability on Trees and Networks}, vol.~42 of Cambridge Series in Statistical and Probabilistic Mathematics,
\newblock Cambridge University Press, New York, 2016, Available at \url{http://pages.iu.edu/~rdlyons/}.

\bibitem{McG02} 
\newblock I.~McGillivray,
\newblock \doititle{Resistance in higher-dimensional {S}ierpi\'nski carpets},
\newblock \emph{Potential Anal.}, \textbf{16} (2002), 289--303.

\bibitem{Str1} 
\newblock D.~Molitor, N.~Ott and R.~Strichartz,
\newblock \doititle{Using {P}eano curves to construct {L}aplacians on fractals},
\newblock \emph{Fractals}, \textbf{23} (2015), 1550048, 29pp.

\bibitem{Nek05} 
\newblock V.~Nekrashevych,
\newblock \emph{Self-similar Groups}, vol. 117 of Mathematical Surveys and Monographs,
\newblock Amer. Math. Soc., 2005.

\bibitem{NT08} 
\newblock V.~Nekrashevych and A.~Teplyaev,
\newblock \doititle{Groups and analysis on fractals},
\newblock in \emph{Analysis on Graphs and Its Applications}, vol.~77 of Proc. Sympos. Pure Math., Amer. Math. Soc., 2008, 143--180.

\bibitem{RT} 
\newblock L.~Rogers and A.~Teplyaev,
\newblock \doititle{Laplacians on the basilica {J}ulia sets},
\newblock \emph{Commun. Pure Appl. Anal.}, \textbf{9} (2010), 211--231.

\bibitem{SteinhurstPOTA} 
\newblock B.~Steinhurst,
\newblock \doititle{Uniqueness of locally symmetric {B}rownian motion on {L}aakso spaces},
\newblock \emph{Potential Anal.}, \textbf{38} (2013), 281--298.

\bibitem{Strbook} 
\newblock R.~S. Strichartz,
\newblock \emph{Differential Equations on Fractals. A Tutorial},
\newblock Princeton Univ. Press, 2006.

\bibitem{Telcs} 
\newblock A.~Telcs and V.~Vespri,
\newblock \doititle{Resolvent metric and the heat kernel estimate for random walks},
\newblock \emph{Stochastic Process. Appl.}, \textbf{124} (2014), 3965--3985.

\bibitem{Tep08} 
\newblock A.~Teplyaev,
\newblock \doititle{Harmonic coordinates on fractals with finitely ramified cell structure},
\newblock \emph{Canad. J. Math.}, \textbf{60} (2008), 457--480.

\bibitem{Volkov} 
\newblock S.~Volkov,
\newblock \doititle{Random geometric subdivisions},
\newblock \emph{Random Structures Algorithms}, \textbf{43} (2013), 115--130.


\end{thebibliography}

\end{document}